\documentclass[11pt,reqno]{amsart}

\usepackage[foot]{amsaddr}
\usepackage{nicefrac, xfrac}
\usepackage{enumerate}
\usepackage{amsfonts, amsmath, amssymb, amscd, amsthm, bm}

\usepackage{enumerate}
\usepackage{url}
\usepackage[linktocpage=true,colorlinks,citecolor=magenta,linkcolor=blue,urlcolor=magenta]{hyperref}
\usepackage{multicol}
\usepackage{comment}
\usepackage{xcolor}

\newtheorem{thm}{Theorem}[section]

\theoremstyle{definition}
\newtheorem{defn}[thm]{Definition}

\numberwithin{equation}{section}
\newtheorem{theorem}{theorem}
\newtheorem{lemma}{lemma}

\numberwithin{equation}{section}

\newcounter{rom}
\renewcommand{\therom}{(\roman{rom})}
	{\end{list}}

\usepackage{amsmath}
\usepackage{amssymb}
\usepackage{cite}

\begin{document}
	\title{Comparison Theorems for the Profile Curve Equation of Rotationally Symmetric Self-Shrinkers}
	\author{Peng Peng}
	\address{School of Mathematical Sciences, Fudan University,
		Shanghai 200433, P.R. China} 
	\email{20110180011@fudan.edu.cn}
	\begin{abstract}
	Mean curvature flow is a fundamental geometric evolution equation in which a submanifold moves in the normal direction with velocity equal to its mean curvature vector. Self-shrinkers arise naturally as self-similar solutions to the mean curvature flow and play an important role as models for finite-time singularities. Among nontrivial examples of compact embedded self-shrinkers, the rotationally symmetric self-shrinking torus constructed by Angenent is one of the most important. However, the uniqueness of the Angenent torus remains a major open problem.
	
	In this paper, we study rotationally symmetric self-shrinkers of type $\mathbb{S}^{1}\times \mathbb{S}^{n-1}$ from the point of view of ordinary differential equations. We analyze the profile curves of rotationally symmetric self-shrinkers, focusing on the behavior of their vertical points and the curves traced out by these points as the initial height varies. We give a new proof of the existence of the Angenent torus by showing that two families of vertical-point trajectories must intersect. 
	
	We further derive the linearized equation associated with the rotationally symmetric self-shrinker equation and apply a Sturm-type comparison theorem to obtain sufficient conditions for the monotonicity of horizontal-point trajectories. In particular, we prove a comparison theorem for solutions near the spherical self-shrinker $x^{2}+r^{2}=2n$, and establish partial monotonicity results for the curves of horizontal points. These results provide a possible approach to the uniqueness problem for the Angenent torus.
	\end{abstract}
	\maketitle
\pagestyle{plain}

\section{Introduction}

Mean curvature flow is one of the most important geometric evolution equations. Let
\[
X_t:M_t\to\mathbb{R}^{n+1}
\]
be a smooth one-parameter family of immersed hypersurfaces, and let $\vec H_t$ denote the mean curvature vector of $X_t(M_t)$. We say that $M_t$ evolves by mean curvature flow if
\begin{equation}\label{eq:mcf-intro}
	\frac{d}{dt}X_t=-\vec H_t.
\end{equation}
Thus, under the mean curvature flow, each point of the hypersurface moves in the normal direction with speed equal to the mean curvature. This flow can be viewed as the negative gradient flow of the area functional; see, for example, \cite{2015Mean,1986heat}.

Self-shrinkers are central objects in the study of mean curvature flow. They are self-similar solutions and arise naturally as singularity models. Suppose that $M_t\subset\mathbb{R}^{n+1}$ is a mean curvature flow defined for $t\in(-\infty,0)$ and has the form
\[
M_t=\sqrt{-t}\,M_{-1}.
\]
Then $M_t$ is a self-similar shrinking solution. In this case, the hypersurface satisfies the self-shrinker equation
\[
H=-\frac{\langle x,\nu_t\rangle}{2t},
\]
where $H$ is the scalar mean curvature and $\nu_t$ is a choice of unit normal. Equivalently, after taking $t=-1$, a hypersurface $\Sigma\subset\mathbb{R}^{n+1}$ is a self-shrinker if it satisfies
\begin{equation}\label{eq:self-shrinker-intro}
	H=\frac{\langle x,\nu\rangle}{2},
\end{equation}
up to the convention for the choice of normal.

The simplest and most important self-shrinkers in $\mathbb{R}^{n+1}$ include the hyperplane $\mathbb{R}^{n}$, the round sphere $\mathbb{S}^{n}$ of radius $\sqrt{2n}$ centered at the origin, and the generalized cylinders
\[
\mathbb{S}^{k}(\sqrt{2k})\times\mathbb{R}^{n-k}.
\]

Despite these simple examples, the general structure of self-shrinkers is highly nontrivial. In dimension one, self-similar shrinking curves in the plane have been classified through the work of Abresch and Langer and related developments \cite{Abresch1986TheNC}\cite{1987EW}\cite{2010HH}. Taking products of such curves with Euclidean factors gives higher-dimensional examples. In higher dimensions, however, the classification of self-shrinkers is far more complicated.

A fundamental breakthrough was made by Angenent \cite{Angenent1992}, who constructed an embedded rotationally symmetric self-shrinking torus. The profile curve of this torus is symmetric with respect to a line perpendicular to the axis of rotation. In this paper, we refer to this rotationally symmetric self-shrinking torus with the additional reflection symmetry as the \emph{Angenent torus}. The existence of the Angenent torus shows that singularities of mean curvature flow in higher dimensions can be much richer than those in the curve shortening flow. In particular, by using the Angenent torus and the maximum principle, one can show that for $n\geq2$, embedded smooth hypersurfaces may develop singularities that split the evolving hypersurface into disconnected components.

Since Angenent's construction, many further examples of nontrivial self-shrinkers have been found. Nguyen \cite{2017Shrinking} gave a variational construction of the Angenent torus. McGrath \cite{McGrath2015ClosedMC} constructed self-shrinkers in $\mathbb{R}^{2n}$ which are diffeomorphic to
\[
\mathbb{S}^{n-1}\times \mathbb{S}^{n-1}\times \mathbb{S}^{1}
\]
and invariant under the action of $SO(n)\times SO(n)$. Nguyen \cite{Nguyen2006ConstructionOC}\cite{Nguyen2009Construction}\cite{Nguyen2014Construction}, M{\o}ller \cite{Moller2011ClosedSS}, and Kapouleas, Kleene and M{\o}ller \cite{Kapouleas2011MeanCShighgenus} constructed many further examples of nontrivial embedded self-shrinkers. More recently, Riedler \cite{RiedleroskarO2023embedd} constructed embedded self-shrinkers of the form $\mathbb{S}^{1}\times M$, where $M$ is an isoparametric hypersurface in $\mathbb{S}^{n}$ with equal principal curvature multiplicities.

Rotationally symmetric self-shrinkers have been studied extensively. Drugan \cite{DG2013ImmersedS2} constructed an immersed self-shrinking sphere in $\mathbb{R}^{3}$. Kleene and M{\o}ller \cite{Kleene2010SelfshrinkersWA} proved that for every regular cone, there exists a unique self-shrinker asymptotic to that cone, giving rise to self-shrinkers with conical shrinking trumpet ends. Drugan and Kleene \cite{Drugan2013ImmersedS} later proved the existence of infinitely many immersed rotationally symmetric self-shrinkers of various topological types, including $\mathbb{S}^{n}$, $\mathbb{R}^{n}$, $\mathbb{R}\times\mathbb{S}^{n-1}$, and $\mathbb{S}^{1}\times\mathbb{S}^{n-1}$.

In \cite{Kleene2010SelfshrinkersWA}, Kleene and M{\o}ller also obtained a classification of complete embedded rotationally symmetric self-shrinkers. Such a hypersurface must be one of the following:
\begin{enumerate}[(1)]
	\item a hyperplane $\mathbb{R}^{n}$;
	\item the cylinder $\mathbb{R}\times\mathbb{S}^{n-1}$ of radius $\sqrt{2(n-1)}$;
	\item the round sphere $\mathbb{S}^{n}$ of radius $\sqrt{2n}$;
	\item a torus diffeomorphic to $\mathbb{S}^{1}\times\mathbb{S}^{n-1}$.
\end{enumerate}
The fourth case leaves open an important question: whether the embedded rotationally symmetric self-shrinking torus is unique. This is one of the central open problems in the study of rotationally symmetric self-shrinkers; see also \cite{2017survey}. A recent result of Mramor \cite{2020Compactness} shows, using the analytic theory of ordinary differential equations, that among tori symmetric with respect to a hyperplane perpendicular to the axis of rotation, there are only finitely many Angenent-type tori. A related finiteness result was also obtained in \cite{Ma2022EntropyBC}. 

There are also important uniqueness and rigidity results beyond the rotationally symmetric setting. Brendle \cite{2014brendle} proved that any compact embedded genus-zero self-shrinker in $\mathbb{R}^{3}$ must be the round sphere. Mramor and Wang \cite{2020MW}, inspired by Lawson's work on embedded minimal surfaces in $\mathbb{S}^{3}$ \cite{1970lawson}, proved that embedded self-shrinking tori in $\mathbb{R}^{3}$ are unknotted. In another direction, Wang \cite{Wang2011} proved that two-dimensional translating solitons in $\mathbb{R}^{3}$ must be rotationally symmetric under suitable hypotheses, and constructed non-rotationally symmetric examples in higher dimensions. More recently, Chu and Sun \cite{2025sun} proved the existence of a non-rotationally symmetric embedded self-shrinker of genus one in $\mathbb{R}^{3}$, although it is not yet known whether this surface is compact.

The goal of this paper is to study the Angenent torus and related rotationally symmetric self-shrinkers from the viewpoint of the ordinary differential equations satisfied by their profile curves. A rotationally symmetric hypersurface in $\mathbb{R}^{n+1}$ can be generated by rotating a curve in the $(x,r)$-plane about the $x$-axis. Then the arc-length profile curve  satisfies the following system of ordinary differential equations:
\begin{equation}\label{eq:arclength-system}
	\begin{cases}
		\dot{x}=\cos\theta, \\
		\dot{r}=\sin\theta, \\
		\dot\theta=\left(\frac{n-1}{r}-\frac{r}{2}\right)\cos\theta+\frac{x}{2}\sin\theta.
	\end{cases}
\end{equation}
Here $\theta$ denotes the angle between the tangent direction of the solution curve and the positive direction of the $x$-axis. If $r$ is regarded as a function of $x$, namely $r=u(x)$, then $u$ satisfies
\begin{equation}\label{eq:u-equation}
	\frac{d\theta}{dx}
	=
	\frac{u''}{1+(u')^{2}}
	=
	\frac{n-1}{u}-\frac{u}{2}+\frac{x}{2}u'.
\end{equation}
If $x$ is regarded as a function of $r$, namely $x=g(r)$, then $g$ satisfies
\begin{equation}\label{eq:g-equation}
	\frac{d\varphi}{dr}
	=
	\frac{g''}{1+(g')^{2}}
	=
	\left(\frac{r}{2}-\frac{n-1}{r}\right)g'-\frac{g}{2}.
\end{equation}
Here $\varphi$ denotes the angle between the solution curve and the positive direction of the $r$-axis.

In this thesis, we will frequently use the following initial value problem:
\begin{equation}\label{ivp:uR}
	\begin{cases}
		\displaystyle
		\frac{d\theta}{dx}
		=
		\frac{u''}{1+(u')^{2}}
		=
		\frac{n-1}{u}-\frac{u}{2}+\frac{x}{2}u', \\[0.3em]
		u(0)=R,\qquad u'(0)=0.
	\end{cases}
\end{equation}
The corresponding solution will be denoted by $u_R$. For convenience, we denote by $\Gamma_R$ the solution of \eqref{eq:arclength-system} with the following initial conditions:
\begin{equation}\label{ivp:GammaR}
	\begin{cases}
		x(0)=0,\\
		r(0)=R,\\
		\theta(0)=0.
	\end{cases}
\end{equation}
If the image of $\Gamma_R$ is viewed as the graph of a function of $r$, then we denote this function by $g_R$.

\begin{defn}
	For the graph $r=u(x)$, a point is called a \emph{horizontal point of $u$} if $u'=0$ there, and a \emph{vertical point of $u$} if $|u'|=\infty$ there. Similarly, for the graph $x=g(r)$, a point is called a \emph{horizontal point of $g$} if $g'=0$ there, and a \emph{vertical point of $g$} if $|g'|=\infty$ there.
\end{defn}

When a solution $g$ or $u$ is mentioned in this thesis, it usually refers to the portion of the solution satisfying $g\geq0$ or $u\geq0$. This thesis begins with the study of the vertical points of the solutions $u_R$ to the initial value problem \eqref{ivp:uR}. These vertical points depend on $R$ and will usually be denoted by $(b_R,u_R(b_R))$, that is, $u_R'(b_R)=\pm\infty$.

As $R$ varies from $+\infty$ to $\sqrt{2n}$, the vertical points form a curve. As $R$ varies from $0$ to $\sqrt{2(n-1)}$, the vertical points form another curve. If these two curves intersect in a suitable way, then one may obtain a solution of \eqref{eq:arclength-system} which is symmetric with respect to the $r$-axis and diffeomorphic to $\mathbb S^1$. Rotating this curve about the $x$-axis gives an Angenent torus. Therefore, studying the properties of the curves formed by the vertical points of solutions is a natural approach to the uniqueness problem.

First, we establish several qualitative properties of solutions to the rotationally symmetric self-shrinker equation. We prove concavity and convexity propagation results, continuity of vertical points under variation of the initial height, and asymptotic estimates for the vertical points as $R\to+\infty$ and $R\to0$. These results allow us to construct the two vertical-point trajectories and prove that they intersect. This gives a new proof of the existence of the Angenent torus.

Second, we derive and study a linearized equation associated with the angle function of solutions to \eqref{eq:g-equation}. Let
\[
\varphi=\arctan g'.
\]
By moving the initial point and differentiating the corresponding angle equation, we obtain a second-order linear equation for the linearized function $h$. We then apply a Sturm-type comparison theorem to obtain sufficient conditions ensuring $h>0$. This positivity implies a comparison principle for nearby profile curves.

Third, we apply the comparison theorem to the study of monotonicity of the horizontal-point trajectories. In particular, we introduce an auxiliary quantity $\eta$ whose nonnegativity implies positivity of the linearized function. We verify this condition near the spherical self-shrinker $x^{2}+r^{2}=2n$, obtaining a comparison theorem for solutions near the sphere. We also prove partial monotonicity results for the horizontal-point curves $L_1$ and $L_2$. These results do not settle the uniqueness problem for the Angenent torus, but they provide a possible route toward it by reducing certain local monotonicity questions to explicit inequalities involving the profile-curve ODE.

\section{Analysis of Convex and Concave Solutions}

\subsection{Continuity of the Vertical Points of Solutions}

\begin{lemma}\label{lem:maximum-principle}(\cite{Kleene2010SelfshrinkersWA})
	Let $r=u(x)$ and $x=g(r)$ be solutions of \eqref{eq:u-equation} and \eqref{eq:g-equation}, respectively.
	
	\begin{enumerate}[(i)]
		\item If there exists $x_0$ such that $u'(x_0)=0$, then $u$ has a local maximum at $x_0$ when $u(x_0)>\sqrt{2(n-1)}$, and a local minimum at $x_0$ when $u(x_0)<\sqrt{2(n-1)}$. If $u(x_0)=\sqrt{2(n-1)}$, then $u$ is the constant cylindrical solution.
		
		\item If there exists $r_0$ such that $g'(r_0)=0$ and $g(r_0)>0$, then $g$ has a local maximum at $r_0$.
	\end{enumerate}
\end{lemma}

\begin{proof}
	Taking $u'(x_0)=0$ in \eqref{eq:u-equation}, we obtain $u''(x_0)=\frac{n-1}{u(x_0)}-\frac{u(x_0)}{2}$. Thus $u''(x_0)<0$ if $u(x_0)>\sqrt{2(n-1)}$, and $u''(x_0)>0$ if $u(x_0)<\sqrt{2(n-1)}$. If $u(x_0)=\sqrt{2(n-1)}$, then the uniqueness theorem for ordinary differential equations implies that $u\equiv\sqrt{2(n-1)}$.
	
	Taking $g'(r_0)=0$ in \eqref{eq:g-equation}, we obtain $g''(r_0)=-\frac{g(r_0)}{2}<0$. Thus $g$ has a local maximum at $r_0$.
\end{proof}

The following lemma is equivalent to Lemma 1 in \cite{Drugan2013ImmersedS}. The formulation and proof below will be convenient for our purposes.

\begin{lemma}\label{lem:concavity-propagation}
	\begin{enumerate}[(i)]
		\item Let $u$ be a nonnegative solution of \eqref{eq:u-equation}. Suppose that $u'<0$. If there exists $x_1$ such that $u''(x_1)<0$, then $u''(x)<0$ for all $x<x_1$ for which $u$ is defined and the assumptions hold. Similarly, if $u'>0$ and there exists $x_1$ such that $u''(x_1)>0$, then $u''(x)>0$ for all $x<x_1$ for which $u$ is defined and the assumptions hold.
		
		\item Let $g$ be a nonnegative solution of \eqref{eq:g-equation}. Suppose that $g'>0$. If there exists $r_1$ such that $g''(r_1)<0$, then $g''(r)<0$ for all $r<r_1$ for which $g$ is defined and the assumptions hold. Similarly, if $g'<0$ and there exists $r_1$ such that $g''(r_1)<0$, then $g''(r)<0$ for all $r>r_1$ for which $g$ is defined and the assumptions hold.
	\end{enumerate}
\end{lemma}

\begin{proof}
	We prove the first assertion in (i). Suppose, by contradiction, that the assertion is false. Then there exists $x_2<x_1$ such that $u''(x_2)=0$, and $u''<0$ on $(x_2,x_1)$. Differentiating \eqref{eq:u-equation}, we obtain
	\begin{equation*}
		u'''
		=
		2u'u''
		\left(
		\frac{n-1}{u}-\frac{u}{2}+\frac{x}{2}u'
		\right)
		+
		(1+(u')^{2})
		\left(
		-\frac{n-1}{u^{2}}u'
		+\frac{x}{2}u''
		\right).
	\end{equation*}
	Evaluating this identity at $x_2$, we get
	\[
	u'''(x_2)
	=
	-(1+(u')^{2})u'\frac{n-1}{u^{2}}(x_2)>0,
	\]
	because $u'<0$. This contradicts the fact that $u''<0$ immediately to the right of $x_2$. The case $u'>0$ and $u''(x_1)>0$ is proved similarly.
	
	For (ii), differentiating \eqref{eq:g-equation} gives
	\begin{equation}\label{eq:g-third-derivative}
		g'''
		=
		\left[
		\frac{n-1}{r^{2}}g'
		+
		\left(\frac{r}{2}-\frac{n-1}{r}\right)g''
		\right](1+(g')^{2})
		+
		2g'g''
		\left[
		\left(\frac{r}{2}-\frac{n-1}{r}\right)g'
		-\frac{g}{2}
		\right].
	\end{equation}
	If $g'>0$ and $g''(r_1)<0$, the same crossing argument as above shows that $g''$ cannot vanish at a point $r_0<r_1$ while changing from positive to negative. Hence $g''<0$ for all $r<r_1$ under the assumptions.
	
	If $g'<0$ and $g''(r_1)<0$, suppose that $g''$ vanishes for the first time at some $r_0>r_1$. Then $g''<0$ on $(r_1,r_0)$ and $g''(r_0)=0$. Evaluating \eqref{eq:g-third-derivative} at $r_0$, we obtain
	 $$g'''(r_0)=(1+(g')^2)\frac{n-1}{r_0^2}g'(r_0)<0,$$ 
	 because $g'<0$. This contradicts the possibility that $g''$ crosses from negative to nonnegative at $r_0$. Therefore $g''<0$ for all $r>r_1$ under the assumptions.
\end{proof}

The following lemma establishes the continuity of vertical points. A further discussion of part (i) requires understanding the behavior of the solutions as $R\to+ \infty$, and also the behavior of solutions near the sphere $u=\sqrt{2n-x^2}$. These will be addressed in Lemma~\ref{lem:large-R-vertical-point} and Section~4.3. For part (ii), Lemma~\ref{lem:small-R-vertical-point} will establish the existence of vertical points of $u_R$ for $R$ near $0$. Similar results can also be found in \cite{Drugan2013ImmersedS}.

\begin{lemma}\label{lem:vertical-point-continuity}
	\begin{enumerate}[(i)]
		\item Suppose that $R_1>R_2>\sqrt{2n}$. Let $u_{R_1}$ and $u_{R_2}$ be solutions of the initial value problem \eqref{ivp:uR}. Assume that there exist $b_{R_1}>0$ and $b_{R_2}>0$ such that $u'_{R_1}(b_{R_1})=u'_{R_2}(b_{R_2})=-\infty$. Then for every $R\in(R_2,R_1)$, there exists $b_R>0$ such that $u'_R(b_R)=-\infty$, and $u''_R(x)<0$ on the corresponding branch $0<x<b_R$. Moreover, the vertical point $(b_R,u_R(b_R))$ depends continuously on $R$, unless the vertical points reach the axis $r=0$ or escape to infinity.
		
		\item Suppose that $0<R_2<R_1<\sqrt{2(n-1)}$. Let $u_{R_1}$ and $u_{R_2}$ be solutions of the initial value problem \eqref{ivp:uR}. Assume that there exist $b_{R_1}>0$ and $b_{R_2}>0$ such that $u'_{R_1}(b_{R_1})=u'_{R_2}(b_{R_2})=+\infty$. Then for every $R\in(R_2,R_1)$, there exists $b_R>0$ such that $u'_R(b_R)=+\infty$, and $u''_R(x)>0$ on the corresponding branch $0<x<b_R$. Moreover, the vertical point $(b_R,u_R(b_R))$ depends continuously on $R$, unless the vertical points escape to infinity.
	\end{enumerate}
\end{lemma}

\begin{proof}
	We prove (i). The proof of (ii) is analogous.
	
	First, we show that the assertion holds for all $R$ sufficiently close to $R_1$ from below. Since $(b_{R_1},u_{R_1}(b_{R_1}))$ is a vertical point of $u_{R_1}$, it is a horizontal point of the corresponding function $g_{R_1}$. By Lemma~\ref{lem:maximum-principle}, this point is a local maximum point of $g_{R_1}$. In particular, $g_{R_1}'=0$ and $g_{R_1}''<0$ at this point.
	
	By the continuous dependence of solutions of ordinary differential equations on initial data, for every $R$ sufficiently close to $R_1$, the function $g_R$ has a nearby critical point. Since this critical point is nondegenerate, it depends continuously on $R$. Moreover, at this critical point one has $g_R''<0$, and hence it is a local maximum point of $g_R$. Equivalently, $u_R$ has a vertical point $b_R$ near $b_{R_1}$.
	
	We now determine the sign of $u_R''$ on the branch before this vertical point. Near the vertical point, but away from the point itself, the two graph representations are related by $g_R'=\frac{1}{u_R'}$ and
	\[
	g_R''
	=
	-\frac{u_R''}{(u_R')^3}.
	\]
	Equivalently,
	\[
	u_R''
	=
	-\frac{g_R''}{(g_R')^3}.
	\]
	In the present case $u_R'\to-\infty$ as $x\to b_R$, and therefore $g_R'<0$ near the corresponding horizontal point of $g_R$. Since $g_R''<0$, it follows that $u_R''<0$ near $b_R$ on the branch $0<x<b_R$. By Lemma~\ref{lem:concavity-propagation}, this negativity propagates backward along the branch. Hence $u_R''(x)<0$ for all $0<x<b_R$. Therefore the assertion holds on some interval $(R_1-\delta,R_1)$.
	
	We next continue this conclusion to the whole interval $(R_2,R_1)$. Let
	\[
	\mathcal I
	=
	\left\{
	\widetilde R\in(R_2,R_1):
	\text{ the assertion holds for all } R\in(\widetilde R,R_1)
	\right\}.
	\]
	The previous paragraph shows that $\mathcal I$ is nonempty. Let $R_*=\inf \mathcal I$. We claim that $R_*=R_2$. Suppose, by contradiction, that $R_*>R_2$. By the continuous dependence of solutions on initial data, the vertical points corresponding to $R\downarrow R_*$ have a limit, unless they reach the axis $r=0$ or escape to infinity. Excluding these exceptional cases, the limiting point is a horizontal point of $g_{R_*}$, and hence a vertical point of $u_{R_*}$. As above, this point is nondegenerate as a critical point of $g_{R_*}$, and therefore it persists for all $R$ sufficiently close to $R_*$. Applying the previous local argument around $R_*$, we extend the assertion to an interval $(R_*-\varepsilon,R_1)$, contradicting the definition of $R_*$. Thus $R_*=R_2$, and the assertion holds for all $R\in(R_2,R_1)$.
	
	The continuity of the vertical point $(b_R,u_R(b_R))$ follows from the same nondegeneracy and the continuous dependence of solutions on initial data. This proves (i).
	
	For (ii), the argument is the same except that $u_R'\to+\infty$ at the vertical point. Hence $g_R'>0$ near the corresponding horizontal point of $g_R$. Since $g_R''<0$, the relation 
	$$u_R''=-\frac{g_R''}{(g_R')^3},$$ 
	gives $u_R''>0$ near $b_R$. Lemma~\ref{lem:concavity-propagation} then implies $u_R''(x)>0$ on the whole branch $0<x<b_R$. The remaining continuity argument is identical.
\end{proof}

\begin{lemma}\label{lem:vertical-point-curve}
	If the vertical points of the solutions $u_R$ in Lemma~\ref{lem:vertical-point-continuity} remain in a bounded region with $r>0$, then these vertical points form a continuous curve in that region.
\end{lemma}

\begin{proof}
	The continuity of the vertical points has already been established in Lemma~\ref{lem:vertical-point-continuity}. It remains to show that the points $(b_R,u_R(b_R))$, $u'_R(b_R)=-\infty$, form a single continuous curve. If there were another such curve, then for some value of $R$, the corresponding function $g_R$ would have two horizontal points. This contradicts Lemma~\ref{lem:concavity-propagation}, which implies that $g_R$ can have at most one such horizontal point on the relevant branch. Hence the vertical points form a continuous curve.
\end{proof}

\subsection{Initial Positions of the Vertical Points}

We first prove a result similar to Lemma 1 in \cite{Angenent1992}, but using a different argument in place of the rescaling technique.

\begin{lemma}\label{lem:large-R-estimate}
	Let $(x_R(t),r_R(t))$ be the solution of \eqref{eq:arclength-system} with initial condition \eqref{ivp:GammaR}, and regard $\theta$ as a function of $r$. Then, for $R$ sufficiently large, $\theta_R\left(R-\frac{1}{R}\right)$ converges to a fixed negative constant $\theta_1$. Moreover, there exist constants $C_1,C_2>0$, depending only on $n$, such that
	\[
	\frac{C_1}{R}
	\leq
	g_R\left(R-\frac{1}{R}\right)
	\leq
	\frac{C_2}{R}.
	\]
\end{lemma}

\begin{proof}
	We divide the proof into three steps.
	
	\medskip
	
	\noindent
	\textbf{Step 1: An upper estimate for $\theta_R(R-\frac1R)$.}
	
	By the reflection symmetry of \eqref{eq:g-equation} with respect to the $r$-axis, we first consider the reflected branch. For notational convenience, we denote this reflected situation formally by $R\to-\infty$. From the initial condition, we have $\varphi_R(R)=\frac{\pi}{2}$. By \eqref{eq:g-equation},
	\begin{align*}
		\frac{d\varphi_R}{dr}
		&<
		\left(\frac{r}{2}-\frac{n-1}{r}\right)g'_R
		=
		\left(\frac{r}{2}-\frac{n-1}{r}\right)\tan\varphi_R.
	\end{align*}
	Hence
	\[
	\frac{d\varphi_R}{dr}\frac{\cos\varphi_R}{\sin\varphi_R}
	<
	\frac{r}{2}-\frac{n-1}{r}.
	\]
	Integrating this inequality from $R$ to $R-\frac1R$, and using $\varphi_R(R)=\frac{\pi}{2}$, we obtain
	\begin{equation*}
		\begin{aligned}
			\sin\varphi_R\left(R-\frac1R\right)
			&<
			\left(\frac{R}{R-\frac1R}\right)^{n-1}
			\exp\left[
			\frac14
			\left(R+R-\frac1R\right)
			\left(R-\frac1R-R\right)
			\right]
			\sin\varphi_R(R)\\
			&=
			\left(\frac{R}{R-\frac1R}\right)^{n-1}
			\exp\left[
			\frac14\left(-2+\frac1{R^2}\right)
			\right].
		\end{aligned}
	\end{equation*}
	Therefore, for $R<0$ with $|R|$ sufficiently large,
	\[
	\varphi_R\left(R-\frac1R\right)
	<
	\arcsin\left[
	\left(\frac{R}{R-\frac1R}\right)^{n-1}
	\exp\left(
	\frac14\left(-2+\frac1{R^2}\right)
	\right)
	\right].
	\]
	Since $\varphi_R+\theta_R=\frac{\pi}{2}$, we get
	\[
	\theta_R\left(R-\frac1R\right)
	>
	\frac{\pi}{2}
	-
	\arcsin\left[
	\left(\frac{R}{R-\frac1R}\right)^{n-1}
	\exp\left(
	\frac14\left(-2+\frac1{R^2}\right)
	\right)
	\right].
	\]
	Passing back to the original branch with $R\to+\infty$, we obtain
	\begin{equation}\label{eq:theta-upper-large-R}
		\theta_R\left(R-\frac1R\right)
		<
		\arcsin\left[
		\left(\frac{R}{R-\frac1R}\right)^{n-1}
		\exp\left(
		\frac14\left(-2+\frac1{R^2}\right)
		\right)
		\right]
		-\frac{\pi}{2}.
	\end{equation}
	\medskip

\noindent
\textbf{Step 2: An upper estimate for $g_R(R-\frac1R)$.}

We now consider the case $R\to+\infty$. On the part of the curve where $r>R-\frac1R$, we view the solution as a graph $r=u_R(x)$. By Lemma~\ref{lem:maximum-principle}, we have $u'_R<0$ on this branch. Let $x_1=g_R\left(R-\frac1R\right)$. Using \eqref{eq:u-equation}, we obtain
\begin{align*}
	-\frac{\pi}{2}
	<
	\int_0^{x_1}\frac{d\theta}{dx}\,dx
	&<
	\int_0^{x_1}
	\left(
	\frac{n-1}{u_R}-\frac{u_R}{2}
	\right)\,dx\\
	&<
	\frac{2(n-1)-\left(R-\frac1R\right)^2}
	{2\left(R-\frac1R\right)}
	x_1.
\end{align*}
Since the coefficient on the right-hand side is negative for $R$ sufficiently large, this gives $g_R\left(R-\frac1R\right)=x_1\leq\frac{C_2}{R}$, where $C_2>0$ is a constant depending only on $n$.

\medskip

\noindent
\textbf{Step 3: A lower estimate for $g_R(R-\frac1R)$.}

We first work again on the reflected branch, formally denoted by $R\to-\infty$. We claim that for $R<0$ with $|R|$ sufficiently large,
\begin{equation}\label{eq:gprime-lower-reflected-large-R}
	g'_R\left(R-\frac1R\right)\geq -\frac1R.
\end{equation}
Suppose, to the contrary, that $g'_R\left(R-\frac1R\right)<\frac1R$. Then there exists $s\in(0,1)$ such that $g'_R\left(R-\frac{s}{R}\right)=-\frac1R$. Using the estimate from Step 2 on the reflected branch, we have $g_R\left(R-\frac1R\right)\leq-\frac{C_2}{R}$. Hence, on the interval $\left(R,R-\frac1R\right)$, equation \eqref{eq:g-equation} gives
\begin{align*}
	\frac{d\varphi_R}{dr}
	&>
	\left(\frac{r}{2}-\frac{n-1}{r}\right)\tan\varphi_R
	+
	\frac{C_2}{2R}.
\end{align*}
By the choice of $s$, on $\left(R,R-\frac{s}{R}\right)$ we have $g'_R(r)\geq -\frac1R$. Therefore,
\begin{equation}\label{eq:phi-differential-lower-reflected-large-R}
	\frac{d\varphi_R}{dr}
	\frac{\cos\varphi_R}{\sin\varphi_R}
	>
	\left(\frac{r}{2}-\frac{n-1}{r}\right)
	-
	\frac{C_2}{2}.
\end{equation}
Integrating \eqref{eq:phi-differential-lower-reflected-large-R} from $R$ to $R-\frac{s}{R}$, and using $\varphi_R(R)=\frac{\pi}{2}$, we obtain
\begin{equation}\label{eq:sinphi-lower-reflected-large-R}
	\begin{aligned}
		&\sin\varphi_R\left(R-\frac{s}{R}\right)\\
		&>\left(\frac{R}{R-\frac{s}{R}}\right)^{n-1}
		\exp\left[
		\frac14
		\left(R+R-\frac{s}{R}\right)
		\left(R-\frac{s}{R}-R\right)
		\right]
		\exp\left[
		\frac{-C_2}{2}
		\left(R-\frac{s}{R}-R\right)
		\right]\\
		&=
		\left(\frac{R}{R-\frac{s}{R}}\right)^{n-1}
		\exp\left[
		\frac14\left(-2+\frac{s}{R^2}\right)s
		\right]
		\exp\left[
		\frac{C_2}{2R}s
		\right].
	\end{aligned}
\end{equation}
Letting $R\to-\infty$, we have $$\sin\varphi_R\left(R-\frac{s}{R}\right)>e^{-s/2}.$$ This contradicts 
$$g'_R\left(R-\frac{s}{R}\right)=\frac1R\to0.$$ 
Thus \eqref{eq:gprime-lower-reflected-large-R} holds.

Applying the same argument on the whole interval $\left(R,R-\frac1R\right)$, we obtain
\[
\sin\varphi_R\left(R-\frac1R\right)
\geq
\left(\frac{R}{R-\frac1R}\right)^{n-1}
\exp\left[
\frac14\left(-2+\frac1{R^2}\right)
\right]
\exp\left[
\frac{C_2}{2R}
\right].
\]
Returning to the original branch $R\to+\infty$, this gives
\[
\theta_R\left(R-\frac1R\right)
\geq
\arcsin\left[
\left(\frac{R}{R-\frac1R}\right)^{n-1}
\exp\left(
\frac14\left(-2+\frac1{R^2}\right)
\right)
\exp\left(
-\frac{C_2}{2R}
\right)
\right]
-\frac{\pi}{2}.
\]
Combining this estimate with \eqref{eq:theta-upper-large-R}, we get
\[
\begin{aligned}
	&
	\arcsin\left[
	\left(\frac{R}{R-\frac1R}\right)^{n-1}
	\exp\left(
	\frac14\left(-2+\frac1{R^2}\right)
	\right)
	\right]
	-\frac{\pi}{2}
	\\
	&\qquad >
	\theta_R\left(R-\frac1R\right)
	\\
	&\qquad \geq
	\arcsin\left[
	\left(\frac{R}{R-\frac1R}\right)^{n-1}
	\exp\left(
	\frac14\left(-2+\frac1{R^2}\right)
	\right)
	\exp\left(
	-\frac{C_2}{2R}
	\right)
	\right]
	-\frac{\pi}{2}.
\end{aligned}
\]
Therefore, $\theta_R\left(R-\frac1R\right)\to\theta_1:=\arcsin(e^{-1/2})-\frac{\pi}{2}<0$ as $R\to+\infty$.

It remains to prove the lower bound for $g_R(R-\frac1R)$. Let $x_2=g_R\left(R-\frac1R\right)$. By Step 2, $x_2\leq \frac{C_2}{R}$. Since $\theta_R\left(R-\frac1R\right)\to\theta_1<0$, for $R$ sufficiently large the angle $\theta_R$ is bounded away from both $0$ and $-\frac{\pi}{2}$ on the relevant branch. Using \eqref{eq:u-equation}, we obtain
\begin{align*}
	\theta_R\left(R-\frac1R\right)
	=
	\int_0^{x_2}\frac{d\theta}{dx}\,dx
	&=
	\int_0^{x_2}
	\left(
	\frac{n-1}{u_R}
	-\frac{u_R}{2}
	+\frac{x}{2}u'_R
	\right)\,dx\\
	&>
	\int_0^{x_2}
	\left(
	\frac{n-1}{u_R}
	-\frac{u_R}{2}
	+\frac{C_2}{2R}\tan\theta_1
	\right)\,dx\\
	&>
	\left(
	\frac{2(n-1)-R^2}{2R}
	+
	\frac{C_2}{2R}\tan\theta_1
	\right)x_2.
\end{align*}
Since $\theta_R(R-\frac1R)$ is bounded above by a fixed negative constant and the coefficient on the right-hand side is of order $-R$, it follows that $g_R\left(R-\frac1R\right)=x_2\geq\frac{C_1}{R}$ for $R$ sufficiently large, where $C_1>0$ depends only on $n$. This completes the proof.
\end{proof}

Using arguments similar to Lemma 2.3 and Lemma 2.4 in \cite{peng2025f} or \cite{Angenent1992}, we obtain the following result, which describes the position of the vertical point as $R\to+\infty$.

\begin{lemma}\label{lem:large-R-vertical-point}
For $R$ sufficiently large, the solution $u_R$ of the initial value problem \eqref{ivp:uR} has a vertical point $(b_R,u_R(b_R))$. Moreover, as $R\to+\infty$, $(b_R,u_R(b_R))\to(0,+\infty)$.
\end{lemma}

\begin{lemma}\label{lem:small-R-zero-gprime}
Suppose that $\sqrt{2(n-1)}\geq R>0$, and let $g_R(r)$ be a solution of \eqref{eq:g-equation} satisfying the initial conditions $g(R)\geq 0$ and $g'(R)=+\infty$. Then, when $R$ is sufficiently small, $g'_R$ must have a zero. If this zero is attained at $r=\bar R$, then $\bar R\to0$ as $R\to0$.
\end{lemma}
\begin{proof}
We argue by contradiction. Suppose that for arbitrarily small $R$, there exists $d$ such that $g'_R(r)>0$ on $(R,d)$. The proof is divided into two steps.

\noindent\textbf{Step 1.}
Correspondingly, consider the solution of \eqref{eq:g-equation} satisfying the initial conditions $g(R)=0$ and $g'(R)=+\infty$. By \eqref{eq:g-equation}, we have
\begin{equation}\label{eq:phi-differential-upper-small-R}
	\frac{d\varphi}{dr}\frac{\cos\varphi}{\sin\varphi}
	<
	\frac{r}{2}-\frac{n-1}{r}.
\end{equation}
Integrating this inequality over the interval $(R,2R)$, we obtain
\begin{equation}\label{eq:sinphi-upper-2R-small-R}
	\sin\varphi_R(2R)
	<
	e^{\frac{3}{4}R^2}\left(\frac{1}{2}\right)^{n-1}.
\end{equation}

We now estimate $g_R(2R)$. Let $x_3=g_R(2R)$. Using \eqref{eq:u-equation}, we obtain
\begin{equation*}
	\begin{aligned}
		\frac{\pi}{2}
		&>
		\int_{0}^{x_3}\frac{d\theta_R}{dx}\,dx
		=
		\int_{0}^{x_3}
		\left(
		\frac{n-1}{u_R}-\frac{u_R}{2}+\frac{1}{2}xu'_R
		\right)\,dx\\
		&>
		\int_{0}^{x_3}
		\left(
		\frac{n-1}{u_R}-\frac{u_R}{2}
		\right)\,dx\\
		&>
		\left(
		\frac{n-1}{2R}-\frac{2R}{2}
		\right)x_3.
	\end{aligned}
\end{equation*}
Thus $g_R(2R)<\frac{2R}{2(n-1)-2R^2}\frac{\pi}{2}$. Therefore, when $R$ is sufficiently small, there exists a constant $C_3$, depending only on $n$, such that
\begin{equation}\label{eq:g-upper-2R-small-R}
	g_R(2R)<C_3R.
\end{equation}

\medskip

\noindent\textbf{Step 2.}
On the other hand, the upper bound for $g_R(2R)$ above implies that, when $R$ is sufficiently small, $g_R(r)<g_R(2R)<C_3R$ on the interval $(R,2R)$. By \eqref{eq:g-equation}, we have
\begin{equation}\label{eq:phi-differential-lower-small-R}
	\begin{aligned}
		\frac{d\varphi_R}{dr}
		&>
		\left(\frac{r}{2}-\frac{n-1}{r}\right)
		\frac{\sin\varphi_R}{\cos\varphi_R}
		-\frac{g_R}{2}\\
		&>
		\left(\frac{r}{2}-\frac{n-1}{r}\right)
		\frac{\sin\varphi_R}{\cos\varphi_R}
		-\frac{1}{2}C_3R.
	\end{aligned}
\end{equation}

We first prove by contradiction that $g'_R(2R)>R$ when $R$ is sufficiently small. Suppose that there exists $s<2R$ such that $g'_R(s)=R$. Since $g''<0$, we have $g'_R>R$ on the interval $(R,s)$. Then \eqref{eq:g-equation} gives, on $(R,s)$,
\begin{equation*}
	\frac{d\varphi_R}{dr}
	\frac{\cos\varphi_R}{\sin\varphi_R}
	>
	\frac{r}{2}-\frac{n-1}{r}
	-C_3R\frac{1}{2R}.
\end{equation*}
Integrating this inequality over $(R,s)$, we obtain
\begin{equation}\label{eq:sinphi-lower-s-small-R}
	\sin\varphi_R(s)
	>
	e^{\frac{1}{4}(s^2-R^2)}
	\left(\frac{R}{s}\right)^{n-1}
	e^{-\frac{1}{2}C_3(s-R)}.
\end{equation}
When $R$ is sufficiently small, this contradicts $g'_R(s)=R$. Hence $g'_R(2R)>R$ for $R$ sufficiently small. Using the techniques in \eqref{eq:phi-differential-lower-small-R} and \eqref{eq:sinphi-lower-s-small-R}, we obtain
\begin{equation}\label{eq:sinphi-lower-2R-small-R}
	\sin\varphi_R(2R)
	>
	e^{\frac{3}{4}R^2}
	\left(\frac{1}{2}\right)^{n-1}
	e^{-\frac{1}{2}C_3R}.
\end{equation}
Inequalities \eqref{eq:sinphi-lower-2R-small-R} and \eqref{eq:sinphi-upper-2R-small-R} imply that, as $R\to0$,
\begin{equation}\label{eq:phi-limit-2R-small-R}
	\varphi_R(2R)
	\to
	\arcsin\left[\left(\frac{1}{2}\right)^{n-1}\right].
\end{equation}
Thus, without loss of generality, we may take $\varphi_R(2R)=\arcsin\left[\left(\frac{1}{2}\right)^{n-1}\right]:=\varphi_1$, which is a fixed constant.

Now we can estimate $g_R(2R)$ from below. Equation \eqref{eq:u-equation} gives
\begin{equation*}
	\frac{\pi}{2}-\varphi_1<\int_{0}^{x_4}[\frac{n-1}{R}-\frac{R}{2}+\frac{x}{2}\tan(\frac{\pi}{2}-\varphi_1)]dx.
\end{equation*}
Taking $x_4=g_R(2R)$, then there exists constant $C_4$ such that 
\begin{equation}\label{eq2.13}
	g_R(2R)>C_4R.
\end{equation}

By the assumptions $g_R>0$ and $g'_R>0$, we have $g''_R<0$ for $r>R$. Therefore, for $r>2R$, $\varphi_R(r)<\varphi_R(2R)$, and hence
\begin{equation}\label{eq:cosphi-comparison-small-R}
	\frac{1}{\cos\varphi_R(r)}
	<
	\frac{1}{\cos\varphi_R(2R)}.
\end{equation}

As in the integration over $(R,2R)$ in \eqref{eq:sinphi-lower-s-small-R}, integrating \eqref{eq:phi-differential-upper-small-R} over the interval $(R,r)$ gives
\[
\sin\varphi_R(r)
<
e^{\frac{1}{4}(r^2-R^2)}
\left(\frac{R}{r}\right)^{n-1}.
\]
Combining this with \eqref{eq:cosphi-comparison-small-R}, we obtain the following estimate on $(2R,r)$:
\begin{equation}\label{eq:gprime-upper-small-R}
	g'_R=\frac{\sin\varphi_R}{\cos\varphi_R}<\frac{\sin\varphi_R}{\cos\varphi_1}<\frac{1}{\cos\varphi_1}e^{\frac{1}{4}(r^2-R^2)}(\frac{R}{r})^{n-1}.
\end{equation}

\textbf{Step 2.} Now we prove this lemma. Consider the following function defined on $(\frac{d}{2},d)$,
\begin{equation}
	\tilde{g}_R=\frac{g_R}{g_R(\frac{d}{2})}.
\end{equation}
Together with \eqref{eq2.13}, we have the following estimate on $[\frac{d}{2},d]$
$$\tilde{g}'_R<\frac{C_4}{\cos\varphi_1}e^{\frac{1}{4}(r^2-R^2)}(\frac{R}{r})^{n-2}\to 0,$$ 
as $R\to 0$ for $n>2$. It is also easy to check that $\tilde{g}''_R$ are bounded on $[\frac{d}{2},d]$ by \eqref{eq2.13} and \eqref{eq:gprime-upper-small-R}. We then deduce a similar contradiction, as in Lemma 2.4 of \cite{peng2025f}, for $n>2$. Then there exists $\bar{R}\in(R,d)$ such that $g'_R(\bar{R})=0$.  Since $d$ can be chosen arbitrarily small as $R\to0$, it follows that $\bar{R}\to0$. Moreover, note that $g_R$ attains its maximum at $\bar{R}$, by continuous dependence on $n$,  this lemma holds for $n=2$. 
\end{proof}

The following lemma uses techniques similar to those in Lemma~\ref{lem:large-R-estimate}.

\begin{lemma}\label{lem:small-R-vertical-point}
Suppose that $R\in(0,\sqrt{2(n-1)})$, and let $u_R$ be the solution of the initial value problem \eqref{ivp:uR}. Then, as $R\to0$, the vertical point $(b_R,u_R(b_R))$ of $u_R$ exists, and $(b_R,u_R(b_R))\to(0,0)$.
\end{lemma}

\begin{proof}

For $r<\sqrt{2(n-1)}$, integrating \eqref{eq:gprime-upper-small-R} yields
\begin{equation}\label{eq:g-growth-small-R}
	\begin{aligned}
		g(r)-g(2R)
		&\leq
		\int_{2R}^{r}
		e^{\frac{1}{4}(t^2-R^2)}
		\left(\frac{R}{t}\right)^{n-1}\,dt\\
		&<
		\frac{1}{\cos\varphi_R(2R)}
		R^{n-1}
		e^{\frac{1}{4}(2(n-1)-R^2)}
		\int_{2R}^{r}\frac{1}{t^{n-1}}\,dt.
	\end{aligned}
\end{equation}
		When $n>2$, the integral term above is
\begin{equation*}
	\frac{1}{\cos\varphi_R(2R)}
	R^{n-1}
	e^{\frac{1}{4}(2(n-1)-R^2)}
	\frac{1}{n-2}
	\left[
	R-\left(\frac{R}{r}\right)^{n-1}r
	\right].
\end{equation*}
When $n=2$, the integral term above is
\begin{equation*}
	\frac{1}{\cos\varphi_R(2R)}
	R^{n-1}
	e^{\frac{1}{4}(2(n-1)-R^2)}
	\left[
	R\ln\left(\frac{r}{R}\right)
	\right].
\end{equation*}
Combining this with \eqref{eq:g-upper-2R-small-R} and \eqref{eq:g-growth-small-R}, we obtain that for $r<\sqrt{2(n-1)}$, if $R\to0$, then $g_R(r)\to0$. Moreover, Lemma~\ref{lem:small-R-zero-gprime} proves that, as $R\to0$, $r_R(b_R)\to0$. This proves the lemma.
\end{proof}

\section{Comparison Theorems}

\subsection{Sturm Comparison Theorem}

The Sturm comparison theorem used here \cite{ode} is a contrapositive form of the standard Sturm comparison theorem. In the following lemma, we give proofs under two different assumptions. The first assumption can be regarded as a special case of the second one. Since the first proof does not require an integral representation, it is simpler and is sufficient for proving all the results in this section; however, both its assumptions and conclusion are more restrictive. The second proof shows that the stated condition is only a sufficient condition. This proof is inspired by the integral method in \cite{Kleene2010SelfshrinkersWA}. Other proofs of the Sturm comparison theorem can be found in \cite{2010Singular,ode}.

\begin{lemma}\label{lem:sturm-comparison}
Let $p,\tilde q,\bar q$ be continuous functions, and consider the following two linear homogeneous ordinary differential equations:
\begin{equation}\label{eq:sturm-tilde}
	\tilde h''+p\tilde h'+\tilde q\tilde h=0,
\end{equation}
and
\begin{equation}\label{eq:sturm-bar}
	\bar h''+p\bar h'+\bar q\bar h=0.
\end{equation}
Suppose that a solution $\bar h$ of \eqref{eq:sturm-bar} satisfies $\bar h>0$ on $[r_0,r_1)$, and that a solution $\tilde h$ of \eqref{eq:sturm-tilde} satisfies $\tilde h(r_0)>0$. If one of the following two conditions holds, then $\tilde h>0$ on $(r_0,r_1)$.

\begin{enumerate}[(i)]
	\item $\left(\frac{\tilde h}{\bar h}\right)'(r_0)\geq0$, and $\tilde q<\bar q$ on $(r_0,r_1)$.
	
	\item $\tilde q\leq \bar q$ on $(r_0,r_1)$, and
	\begin{equation}\label{eq:sturm-initial-wronskian}
		\tilde h'(r_0)\bar h(r_0)-\tilde h(r_0)\bar h'(r_0)\geq0.
	\end{equation}
\end{enumerate}
\end{lemma}

\begin{proof}
We first prove the assertion under condition (i). Set $h=\frac{\tilde h}{\bar h}$. Then $\tilde h'=\bar h'h+\bar h h'$ and $\tilde h''=\bar h''h+2\bar h'h'+\bar h h''$. Substituting these identities into \eqref{eq:sturm-tilde}, and using \eqref{eq:sturm-bar}, we obtain
\begin{equation}\label{eq:sturm-ratio-equation}
	h''
	+
	\left(
	\frac{2\bar h'}{\bar h}+p
	\right)h'
	+
	(\tilde q-\bar q)h
	=0.
\end{equation}
By assumption, $h(r_0)>0$ and $h'(r_0)\geq0$. Suppose, by contradiction, that there exists $r_2\in(r_0,r_1)$ such that $\tilde h(r_2)=0$. Since $\bar h>0$, this is equivalent to $h(r_2)=0$. Hence $h$ must attain a positive local maximum at some point in $(r_0,r_2)$. At such a point, we have $h>0$, $h'=0$, and $h''\leq0$. However, \eqref{eq:sturm-ratio-equation} gives $h''=(\bar q-\tilde q)h>0$, which is a contradiction. Therefore $\tilde h>0$ on $(r_0,r_1)$.

We now prove the assertion under condition (ii). Rewrite \eqref{eq:sturm-tilde} as the following nonhomogeneous linear equation:
\begin{equation}\label{eq:sturm-nonhomogeneous}
	\tilde h''+p\tilde h'+\bar q\tilde h
	=
	(\bar q-\tilde q)\tilde h.
\end{equation}
The corresponding homogeneous equation is
\begin{equation}\label{eq:sturm-homogeneous}
	y''+py'+\bar q y=0.
\end{equation}
Clearly, $\bar h$ is a solution of \eqref{eq:sturm-homogeneous}. A second linearly independent solution of \eqref{eq:sturm-homogeneous} is given by
\begin{equation*}
	\bar h_2(r)
	=
	\bar h(r)
	\int_{r_0}^{r}
	\frac{1}{\bar h^2(t)}
	e^{-\int_{r_0}^{t}p(z)\,dz}\,dt.
\end{equation*}
For convenience, set $\bar h_1=\bar h$ and $Q=(\bar q-\tilde q)\tilde h$. The Wronskian of $\bar h_1$ and $\bar h_2$ is
\[
W
=
\bar h_1\bar h_2'-\bar h_1'\bar h_2
=
e^{-\int_{r_0}^{r}p(z)\,dz}.
\]
By variation of constants, a particular solution of \eqref{eq:sturm-nonhomogeneous} is
\begin{equation*}
	h^*
	=
	\bar h_2
	\int_{r_0}^{r}
	\frac{\bar h_1 Q}{W}\,dt
	-
	\bar h_1
	\int_{r_0}^{r}
	\frac{\bar h_2 Q}{W}\,dt.
\end{equation*}
Substituting the expressions for $\bar h_2$ and $W$, and integrating by parts, we obtain
\begin{align*}
	h^*
	&=
	\bar h
	\int_{r_0}^{r}
	\frac{1}{\bar h^2(t)}
	e^{-\int_{r_0}^{t}p(z)\,dz}
	\left[
	\int_{r_0}^{t}
	e^{\int_{r_0}^{s}p(z)\,dz}
	(\bar q-\tilde q)(s)\bar h(s)\tilde h(s)\,ds
	\right]dt.
\end{align*}
Therefore, the solution $\tilde h$ of \eqref{eq:sturm-tilde} can be represented as
\begin{equation*}
	\begin{aligned}
		\tilde h(r)
		&=
		c_1\bar h(r)
		+
		c_2\bar h(r)
		\int_{r_0}^{r}
		\frac{1}{\bar h^2(t)}
		e^{-\int_{r_0}^{t}p(z)\,dz}\,dt\\
		&\quad+
		\bar h(r)
		\int_{r_0}^{r}
		\frac{1}{\bar h^2(t)}
		e^{-\int_{r_0}^{t}p(z)\,dz}
		\left[
		\int_{r_0}^{t}
		e^{\int_{r_0}^{s}p(z)\,dz}
		(\bar q-\tilde q)(s)\bar h(s)\tilde h(s)\,ds
		\right]dt.
	\end{aligned}
\end{equation*}
Using the initial values at $r_0$, we find $c_1=\frac{\tilde h(r_0)}{\bar h(r_0)}$ and
\[
c_2=
\tilde h'(r_0)\bar h(r_0)
-
\tilde h(r_0)\bar h'(r_0).
\]
Thus
\begin{equation}\label{eq:sturm-integral-representation}
	\begin{aligned}
		\tilde h(r)
		&=
		\frac{\tilde h(r_0)}{\bar h(r_0)}\bar h(r)\\
		&\quad+
		\left[
		\tilde h'(r_0)\bar h(r_0)
		-
		\tilde h(r_0)\bar h'(r_0)
		\right]
		\bar h(r)
		\int_{r_0}^{r}
		\frac{1}{\bar h^2(t)}
		e^{-\int_{r_0}^{t}p(z)\,dz}\,dt\\
		&\quad+
		\bar h(r)
		\int_{r_0}^{r}
		\frac{1}{\bar h^2(t)}
		e^{-\int_{r_0}^{t}p(z)\,dz}
		\left[
		\int_{r_0}^{t}
		e^{\int_{r_0}^{s}p(z)\,dz}
		(\bar q-\tilde q)(s)\bar h(s)\tilde h(s)\,ds
		\right]dt.
	\end{aligned}
\end{equation}

The first term in \eqref{eq:sturm-integral-representation} is positive, and the second term is nonnegative by \eqref{eq:sturm-initial-wronskian}. Suppose, by contradiction, that $\tilde h$ has a zero in $(r_0,r_1)$, and let $\bar r$ be its first zero. Then $\tilde h>0$ on $(r_0,\bar r)$. Since $\bar q-\tilde q\geq0$ and $\bar h>0$, the third term in \eqref{eq:sturm-integral-representation} is nonnegative for $r=\bar r$. Therefore, $\tilde h(\bar r)>0$, which contradicts $\tilde h(\bar r)=0$. Hence $\tilde h>0$ on $(r_0,r_1)$.
\end{proof}

\subsection{The Linearized Equation for Rotationally Symmetric Self-Shrinkers}

We now study the linearized equation for the angle $\varphi=\arctan g'$ of a solution to \eqref{eq:g-equation}. For a solution satisfying $g(a)=0$, equation \eqref{eq:g-equation} can be equivalently written as
\begin{equation}\label{eq:angle-equation}
\frac{d\varphi}{dr}
=
\left(\frac{r}{2}-\frac{n-1}{r}\right)\tan\varphi
-\frac{1}{2}\int_{a}^{r}\tan\varphi(s)\,ds .
\end{equation}
This integral form will be convenient for deriving the linearized equation below.

Let $0<C<+\infty$. For each initial point $a$, consider the solution of \eqref{eq:g-equation} satisfying $g(a)=0$ and $g'(a)=C$. Equivalently, we consider the solution of \eqref{eq:angle-equation} satisfying
\begin{equation*}
\begin{cases}
	\varphi(a)=\arctan C,\\
	\varphi'(a)=\left(\dfrac{a}{2}-\dfrac{n-1}{a}\right)C.
\end{cases}
\end{equation*}
Here the value of $\varphi'(a)$ follows directly from \eqref{eq:angle-equation}, since the integral term vanishes at $r=a$. We denote this solution by $\varphi_{a,C}$, and the corresponding solution of \eqref{eq:g-equation} by $g_{a,C}$.

\begin{defn}
The \textbf{linearized function} at $r_0$ is defined by
\begin{equation*}
	h_{r_0,C}(r)
	=
	\lim_{\epsilon\to0}
	\frac{\varphi_{r_0+\epsilon,C}(r)-\varphi_{r_0,C}(r)}{\epsilon}
	=
	\frac{\partial}{\partial a}\bigg|_{a=r_0}\varphi_{a,C}(r).
\end{equation*}

For the case $C=+\infty$, namely for solutions satisfying $g'(a)=+\infty$, one may define the corresponding linearized function $h_{r_0,+\infty}$ by taking the limit as $C\to+\infty$, whenever this limit exists.

In what follows, when we refer to a solution $g$, we only consider the portion on which $g'\geq0$ and $g\geq0$. Similarly, for the linearized function $h$ at $r_0$, we only consider its restriction to these intervals. When there is no danger of confusion, we shall omit one or both subscripts in $g_{r_0,C}$, $h_{r_0,C}$, $g_{r_0,+\infty}$, and $h_{r_0,+\infty}$.
\end{defn}

We now derive the equation satisfied by $h$. Differentiating \eqref{eq:angle-equation} with respect to the initial point $a$, and then setting $a=r_0$, gives
\begin{equation*}
\begin{aligned}
	h'
	&=
	\left(\frac{r}{2}-\frac{n-1}{r}\right)\frac{1}{\cos^2\varphi}h
	+\frac{1}{2}\tan\varphi(r_0)
	-\frac{1}{2}\int_{r_0}^{r}\frac{1}{\cos^2\varphi}h  \\
	&=
	\left(\frac{r}{2}-\frac{n-1}{r}\right)(1+(g')^2)h
	+\frac{C}{2}
	-\frac{1}{2}\int_{r_0}^{r}(1+(g')^2)h .
\end{aligned}
\end{equation*}
Here the term $\frac{C}{2}$ comes from differentiating the lower limit of the integral in \eqref{eq:angle-equation} with respect to the initial point $a$.

Differentiating the above equation with respect to $r$ and simplifying, we obtain the following second-order linear ordinary differential equation for $h$, namely the linearized equation associated with \eqref{eq:angle-equation}:
\begin{equation}\label{eq:linearized-equation}
\begin{aligned}
	h''-
	&\left(\frac{r}{2}-\frac{n-1}{r}\right)(1+(g')^{2})h' \\
	&+
	\left[
	-2\left(\frac{r}{2}-\frac{n-1}{r}\right)g'g''
	-\frac{n-1}{r^{2}}(1+(g')^{2})
	\right]h=0.
\end{aligned}
\end{equation}

We next compute the initial conditions satisfied by $h$. For $0<C<+\infty$, by the analytic theory of ordinary differential equations, the solution of \eqref{eq:g-equation} satisfying $g(r_0)=0$ and $g'(r_0)=C$ admits a convergent power series expansion. Hence Taylor expansions may be used to compute the initial values of $h$.

For $a$ close to $r_0$, we have
\begin{equation*}
\varphi_{a,C}(r)
=
\arctan\left[
g'_{a,C}(a)
+
g''_{a,C}(a)(r-a)
+
O((r-a)^2)
\right].
\end{equation*}
In particular, since $g'_{a,C}(a)=C$, we obtain
\begin{equation*}
\varphi_{r_0+\epsilon,C}(r_0)
=
\arctan\left[
C-\epsilon\, g''_{r_0+\epsilon,C}(r_0+\epsilon)
+
O(\epsilon^2)
\right].
\end{equation*}
Therefore
\begin{equation}\label{eq:h-initial-value}
\begin{aligned}
	h(r_0)
	&=
	\lim_{\epsilon\to0}
	\frac{
		\varphi_{r_0+\epsilon,C}(r_0)
		-
		\varphi_{r_0,C}(r_0)
	}{\epsilon} \\
	&=
	-\frac{1}{1+C^2}g''_{r_0,C}(r_0) \\
	&=
	-\left(\frac{r_0}{2}-\frac{n-1}{r_0}\right)C \\
	&=
	\frac{1}{2}g_{r_0,C}(r_0)
	-
	\left(\frac{r_0}{2}-\frac{n-1}{r_0}\right)g'_{r_0,C}(r_0).
\end{aligned}
\end{equation}
Here we used $g_{r_0,C}(r_0)=0$ and
\[
g''_{r_0,C}(r_0)
=
(1+C^2)
\left(\frac{r_0}{2}-\frac{n-1}{r_0}\right)C.
\]

The value of $h'(r_0)$ can be computed either from \eqref{eq:g-equation} and the Taylor expansions of $g$ and $g'$, or directly from the first-order linearized equation above. Since the integral term vanishes at $r=r_0$, we have
\begin{equation}\label{eq:hprime-initial-value}
\begin{aligned}
	h'(r_0)
	&=
	\left(\frac{r_0}{2}-\frac{n-1}{r_0}\right)(1+C^2)h(r_0)
	+\frac{C}{2} \\
	&=
	\frac{1}{2}C
	-
	(1+C^2)C
	\left(\frac{r_0}{2}-\frac{n-1}{r_0}\right)^2.
\end{aligned}
\end{equation}

For the case $g'(r_0)=+\infty$, the initial values of $h$ cannot be computed directly in the above way. In the next lemma, we will instead obtain the corresponding conclusion by taking the limit $C\to+\infty$.

Notice that, by the definition of the linearized function, if $h>0$ is proved, then for all sufficiently small positive $\epsilon$, the angle functions satisfy $\varphi_{r_0,C}<\varphi_{r_0+\epsilon,C}$. We can now obtain a sufficient condition for $h>0$ from the Sturm comparison theorem. When there is no danger of confusion, we shall omit one or both subscripts of $\eta_{r_0,C}$ or $\eta_{r_0,+\infty}$ in the following lemma.

\begin{lemma}\label{lem:eta-comparison}
For a solution with initial condition $g(r_0)=0$ and $g'(r_0)=C$ or $g'(r_0)=+\infty$, define
\begin{equation}\label{eq:eta-definition}
	\eta_{r_0}=(1+(g'_{r_0})^{2})(\frac{n-1}{r^{2}}+\frac{1}{2})+2[(\frac{r}{2}-\frac{n-1}{r})-(\frac{r_{0}}{2}-\frac{n-1}{r_{0}})]g'_{r_0}g''_{r_0}.
\end{equation}

({\romannumeral1}): For $r_{0}\in(-\infty,-\sqrt{2(n-1)})$, or $r_{0}\in (0,\sqrt{2(n-1)})$, suppose that $g_{r_0,C}$ is the solution of \eqref{eq:g-equation} satisfying $g(r_{0})=0$, $g'(r_{0})=C>0$. If $\eta_{r_0,C}\geq0$ on the portion where $g\geq0$ and $g'\geq0$, then the corresponding solution $h_{r_0,C}$ of the linearized equation \eqref{eq:linearized-equation} satisfies $h_{r_0,C}>0$.

({\romannumeral2}): When $g(r_{0})=0$, $g'(r_{0})=+\infty$, if the solution $g_{r_0,+\infty}$ satisfies $\eta_{r_0,+\infty}>0$ on the portion where $g\geq0$ and $g'\geq0$, then the corresponding solution of the linearized equation \eqref{eq:linearized-equation} satisfies $h_{r_0,+\infty}>0$.

({\romannumeral3}): If, for arbitrarily large $C$, the corresponding solution $g_{r_0,C}$ satisfies $\eta_{r_0,C}\geq0$ on the portion where $g\geq0$ and $g'\geq0$, then $h_{r_0,+\infty}>0$.
\end{lemma}

\begin{proof}
Proof of ({\romannumeral1}): By the initial values \eqref{eq:h-initial-value} and \eqref{eq:hprime-initial-value}, define the auxiliary comparison function
\begin{equation}\label{eq:auxiliary-hbar}
	\bar{h}=\frac{1}{2}g-\left(\frac{r_{0}}{2}-\frac{n-1}{r_{0}}\right)g'.
\end{equation}
Then $\bar{h}(r_{0})=h(r_{0})$ and $\bar{h}'(r_{0})=h'(r_{0})$, and hence
\begin{equation*}
	h'(r_{0})\bar{h}(r_{0})-h(r_{0})\bar{h}'(r_{0})=0.
\end{equation*}
Since $\frac{r_{0}}{2}-\frac{n-1}{r_{0}}<0$ under the assumptions $r_{0}\in(-\infty,-\sqrt{2(n-1)})$ or $r_{0}\in(0,\sqrt{2(n-1)})$, and since we only consider the portion of the solution where $g\geq0$ and $g'\geq0$, we have $\bar{h}>0$.

Next we construct the equation satisfied by $\bar{h}$. Suppose that $\bar{h}$ satisfies an equation of the form
\begin{equation}\label{eq:hbar-equation}
	\bar{h}''-\left(\frac{r}{2}-\frac{n-1}{r}\right)(1+(g')^{2})\bar{h}'+\bar{q}\bar{h}=0.
\end{equation}
	Substituting $\bar{h}$, $\bar{h}'$, and
\begin{equation*}
	\begin{aligned}
		\bar{h}''
		&=\frac{1}{2}g''-\left(\frac{r_{0}}{2}-\frac{n-1}{r_{0}}\right)g'''\\
		&=\frac{1}{2}g''
		-\left(\frac{r_{0}}{2}-\frac{n-1}{r_{0}}\right)
		\bigg\{
		\frac{n-1}{r^{2}}(1+(g')^{2})g'\\
		&\quad+
		\left(\frac{r}{2}-\frac{n-1}{r}\right)(1+(g')^{2})g''
		+2g'g''\left[\left(\frac{r}{2}-\frac{n-1}{r}\right)g'-\frac{g}{2}\right]
		\bigg\},
	\end{aligned}
\end{equation*}
into \eqref{eq:hbar-equation}, we obtain
\begin{equation*}
	\begin{aligned}
		\bar{q}
		&=\frac{1}{\bar{h}}\bigg[
		2\left(\frac{r_{0}}{2}-\frac{n-1}{r_{0}}\right)
		\left(\frac{r}{2}-\frac{n-1}{r}\right)(g')^{2}g''
		+\left(\frac{r_{0}}{2}-\frac{n-1}{r_{0}}\right)\frac{n-1}{r^{2}}(1+(g')^{2})g'\\
		&\quad+\frac{1}{4}(1+(g')^{2})g
		-\left(\frac{r_{0}}{2}-\frac{n-1}{r_{0}}\right)gg'g''
		\bigg].
	\end{aligned}
\end{equation*}

Let $q$ denote the coefficient of the $h$-term in the linearized equation \eqref{eq:linearized-equation}. A direct computation gives
\begin{equation}\label{eq:qbar-minus-q}
	\begin{aligned}
		\bar{q}-q
		&=\bar{q}+2\left(\frac{r}{2}-\frac{n-1}{r}\right)g'g''
		+\frac{n-1}{r^{2}}(1+(g')^{2})\\
		&=\frac{1}{2\bar{h}}g\bigg[
		(1+(g')^{2})\left(\frac{n-1}{r^{2}}+\frac{1}{2}\right)\\
		&\quad+
		2\left[
		\left(\frac{r}{2}-\frac{n-1}{r}\right)
		-\left(\frac{r_{0}}{2}-\frac{n-1}{r_{0}}\right)
		\right]g'g''
		\bigg]\\
		&=\frac{g}{2\bar{h}}\eta.
	\end{aligned}
\end{equation}
When $\bar{q}-q\geq0$, equivalently when $\eta\geq0$, the Sturm comparison theorem, Lemma~\ref{lem:sturm-comparison}, implies that $h>0$.

Proof of ({\romannumeral2}):

$\mathbf{Step\ 1:}$ We first compute the limit of $\eta_{r_0,+\infty}$ at $r_0$ in the case $C=+\infty$. For simplicity of notation, in the following computation we write $g_{r_0,+\infty}$ simply as $g$.
\begin{equation}\label{eq:eta-limit-r0-computation}
	\begin{aligned}
		\lim_{r\to r^{+}_{0}} \eta_{r_0,+\infty}
		=&\lim_{r\to r^{+}_{0}}\frac{\frac{1+(g')^{2}}{g'g''}(\frac{n-1}{r^{2}}+\frac{1}{2})+ 2[(\frac{r}{2}-\frac{n-1}{r})-(\frac{r_{0}}{2}-\frac{n-1}{r_{0}})]}{\frac{1}{g'g''}}\\
		=&\lim_{r\to r^{+}_{0}}\frac{\frac{d}{dr}\{\frac{1+(g')^{2}}{g'g''}(\frac{n-1}{r^{2}}+\frac{1}{2})+ 2[(\frac{r}{2}-\frac{n-1}{r})-(\frac{r_{0}}{2}-\frac{n-1}{r_{0}})]\}}{\frac{d}{dr}\{\frac{1}{g'g''}\}}\\
		=&\lim_{r\to r^{+}_{0}} \frac{2(n-1)}{r^3}\frac{g'g''(1+(g')^2)}{g''g''+g'g'''}-4(\frac{n-1}{r^{2}}+\frac{1}{2})\frac{(g')^2(g'')^2}{g''g''+g'g'''}\\
		&+(\frac{n-1}{r^{2}}+\frac{1}{2})(1+(g')^{2})\\
		=&\lim_{r\to r^{+}_{0}} \frac{2(n-1)}{r^3}\frac{g'g''(1+(g')^2)}{g''g''+g'g'''}+\frac{-3(g')^2+1+\frac{g'g'''}{g''g''}+\frac{(g')^3g'''}{g''g''}}{1+\frac{g'g'''}{g''g''}}.
	\end{aligned}
\end{equation}
We compute the above limit in parts. First,
\begin{equation}\label{eq:eta-limit-part1}
	\begin{aligned}
		&\lim_{r\to r^{+}_{0}} -3(g')^2+1+\frac{g'g'''}{g''g''}+\frac{(g')^3g'''}{g''g''}\\
		=&2+(\frac{n-1}{r_0^2}-\frac{1}{2})\frac{1}{(\frac{r_0}{2}-\frac{n-1}{r_0})^2},
	\end{aligned}
\end{equation}
where we have used the limits
$$\lim_{r\to r^{+}_{0}} gg'=\lim_{r\to r^{+}_{0}} \frac{\frac{d}{dr}g}{\frac{d}{dr}\frac{1}{g'}}=\lim_{r\to r^{+}_{0}} \frac{-(g')^3}{g''}=-\frac{1}{\frac{r_0}{2}-\frac{n-1}{r_0}},$$
and
$$\lim_{r\to r^{+}_{0}} \frac{(g')^3g'''}{g''g''}=3.$$
Moreover, a direct computation gives
\begin{equation}\label{eq:eta-limit-part2}
	\lim_{r\to r^{+}_{0}} \frac{2(n-1)}{r^3}\frac{g'g''(1+(g')^2)}{g''g''+g'g'''}=\frac{n-1}{2r_0^3}\frac{1}{\frac{r_0}{2}-\frac{n-1}{r_0}}.
\end{equation}
Substituting \eqref{eq:eta-limit-part1} and \eqref{eq:eta-limit-part2} into \eqref{eq:eta-limit-r0-computation}, we obtain
\begin{equation}\label{eq:eta-limit-r0-positive}
	\begin{aligned}
		\lim_{r\to r^{+}_{0}} \eta_{r_0,+\infty}
		=&\frac{1}{4}[2+(\frac{n-1}{r_0^2}-\frac{1}{2})\frac{1}{(\frac{r_0}{2}-\frac{n-1}{r_0})^2}](\frac{n-1}{r_0^{2}}+\frac{1}{2})+\\
		&\frac{n-1}{2r_0^3}\frac{1}{\frac{r_0}{2}-\frac{n-1}{r_0}}\\
		&=\frac{1}{4r_0^2}[2n-3+r_0^2]>0.
	\end{aligned}
\end{equation}

$\mathbf{Step\ 2:}$ By direct computation, we have
\begin{equation}
	\begin{aligned}
		\frac{d}{dr}\eta_{r_0,C}
		=&2(\frac{n-1}{r^2}+\frac{1}{2})g'g''-\frac{2(n-1)}{r^3}(1+(g')^2)\\
		=&g'g''[2(\frac{n-1}{r^2}+\frac{1}{2})-\frac{2(n-1)}{r^3}\frac{1+(g')^2}{g'g''}]\\
		=&g'g''\big\{2(\frac{n-1}{r^2}+\frac{1}{2})-\frac{2(n-1)}{r^3}\frac{1}{g'[(\frac{r}{2}-\frac{n-1}{r})g'-\frac{g}{2}]}\big\}.
	\end{aligned}
\end{equation}
By the continuous dependence of solutions to the system \eqref{eq:arclength-system} on initial data, when $C$ is sufficiently large, there exists a constant $\delta_1>0$ such that on the interval $(r_0,r_0+\delta_1)$,
$$2(\frac{n-1}{r^2}+\frac{1}{2})-\frac{2(n-1)}{r^3}\frac{1}{g'[(\frac{r}{2}-\frac{n-1}{r})g'-\frac{g}{2}]}>0.$$
Therefore, for sufficiently large $C$, we have $\frac{d}{dr}\eta_{r_0,C}<0$ on the interval $(r_0,r_0+\delta_1)$.

Since the limit \eqref{eq:eta-limit-r0-positive} exists, there exists a constant $\delta_2>0$ such that $\eta_{r_0,+\infty}>0$ on the interval $(r_0,r_0+\delta_2)$. Let $$\bar{\delta}=\min\{\delta_1,\delta_2\}.$$ Since
$$\lim_{C\to+\infty}\eta_{r_0,C}(r_0+\frac{\bar{\delta}}{2})=\eta_{r_0,+\infty}(r_0+\frac{\bar{\delta}}{2})>0,$$
it follows that, for sufficiently large $C$, $\eta_{r_0,C}(r_0+\frac{\bar{\delta}}{2})>0$. Moreover, since $\frac{d}{dr}\eta_{r_0,C}<0$ on the interval $[r_0,r_0+\frac{\bar{\delta}}{2})$, we have $\eta_{r_0,C}>0$ on this interval.

On the set difference between the interval where $g'>0$ holds and $[r_0,r_0+\frac{\bar{\delta}}{2})$, by the differentiable dependence of solutions of ordinary differential equations on initial data in bounded regions, $\eta_{r_0,C}(r)$ converges uniformly to $\eta_{r_0,+\infty}(r)>0$. Hence $\eta_{r_0,C}(r)>0$ holds for sufficiently large $C$.

$\mathbf{Step\ 3:}$ We now prove the conclusion in the case $g'(r_0)=+\infty$. Since $\eta_{r_0,C}(r)>0$ for sufficiently large $C$, it follows that $h_{r_0,C}>0$. Moreover, since $h_{r_0,C}\to h_{r_0,+\infty}$, we obtain $h_{r_0,+\infty}>0$.

Proof of ({\romannumeral3}): This case has already been included in $\mathbf{Step\ 3}$ of the proof of ({\romannumeral2}).
\end{proof}

\section{Application of the Comparison Theorem \expandafter{\romannumeral1}: A Comparison Theorem Near the Sphere}

\subsection{The Linearized Equation Near the Sphere}

In this section, we study the behavior of solutions to the rotationally symmetric self-shrinker equation near the spherical solution $x^{2}+r^{2}=2n$. More precisely, we consider the solutions of \eqref{eq:g-equation} satisfying $g(r_{0})=0$ and $g'(r_{0})=+\infty$, that is, $\varphi(r_{0})=\frac{\pi}{2}$, when $r_{0}$ is close to $-\sqrt{2n}$. We compare the angle function $\varphi$ of such a solution with that of the spherical solution $g(r)=\sqrt{2n-r^{2}}$. By applying Lemma~\ref{lem:eta-comparison} from the previous section, we obtain the conclusion in the following lemma. Namely, when $r_{0}<-\sqrt{2n}$, we have $\varphi_{r_0}<\varphi_{-\sqrt{2n}}$, whereas when $r_{0}>-\sqrt{2n}$, we have $\varphi_{r_0}>\varphi_{-\sqrt{2n}}$.

\begin{lemma}\label{lem:sphere-comparison}
Let $\varphi_{r_0}$ denote the angle function of the solution $g_{r_0}$ to \eqref{eq:g-equation} satisfying $g(r_{0})=0$ and $g'(r_{0})=+\infty$. Let $\varphi_{-\sqrt{2n}}$ denote the angle function of the spherical solution $g(r)=\sqrt{2n-r^{2}}$. Then, when $r_{0}>-\sqrt{2n}$ is sufficiently close to $-\sqrt{2n}$, we have $\varphi_{r_{0}}>\varphi_{-\sqrt{2n}}$. When $r_{0}<-\sqrt{2n}$ is sufficiently close to $-\sqrt{2n}$, we have $\varphi_{r_{0}}<\varphi_{-\sqrt{2n}}$.
\end{lemma}

\begin{proof}
By Lemma~\ref{lem:eta-comparison}, it suffices to verify that the spherical solution $g(r)=\sqrt{2n-r^{2}}$ satisfies $\eta>0$.

For $g(r)=\sqrt{2n-r^{2}}$, we have $g'=\frac{-r}{\sqrt{2n-r^{2}}},$ and $g''=\frac{-2n}{(2n-r^{2})^{\frac{3}{2}}}.$ Substituting these two identities into $\eta$, we obtain
\begin{equation*}
	\begin{aligned}
		\frac{1}{2}\eta
		&=
		\frac{1}{2}\left(\frac{n-1}{r^{2}}+\frac{1}{2}\right)
		\frac{2n}{2n-r^{2}}
		+
		\left(\frac{r}{2}-\frac{n-1}{r}+\frac{1}{\sqrt{2n}}\right)
		\frac{2nr}{(2n-r^{2})^{2}}\\
		&=
		\frac{2n}{2n-r^{2}}
		\left[
		\frac{1}{2}\left(\frac{n-1}{r^{2}}+\frac{1}{2}\right)
		+
		\left(\frac{r}{2}-\frac{n-1}{r}+\frac{1}{\sqrt{2n}}\right)
		\frac{r}{2n-r^{2}}
		\right].
	\end{aligned}
\end{equation*}

We estimate the factor $\frac{r}{2n-r^{2}}$ appearing in the second term on the right-hand side. It is easy to see that the function $\frac{r}{-\sqrt{2n}+r}$ is decreasing, and its value at $r=-\sqrt{2n}$ is $\frac12$. Hence $\frac{r}{-\sqrt{2n}+r}\leq \frac12$. Therefore,
\begin{equation*}
	\frac{r}{2n-r^{2}}
	=
	\frac{1}{-\sqrt{2n}-r}
	\frac{r}{-\sqrt{2n}+r}
	\geq
	\frac{1}{2}\frac{1}{-\sqrt{2n}-r}.
\end{equation*}

Continuing the computation of $\eta$, by the above estimate and the fact that $\frac{r}{2}-\frac{n-1}{r}+\frac{1}{\sqrt{2n}}\geq 0$ for $r$ sufficiently close to $-\sqrt{2n}$, we obtain
\begin{equation*}
	\begin{aligned}
		\frac{1}{2}\eta
		&\geq
		\frac{2n}{2n-r^{2}}
		\left[
		\frac{1}{2}\left(\frac{n-1}{r^{2}}+\frac{1}{2}\right)
		+
		\left(\frac{r}{2}-\frac{n-1}{r}+\frac{1}{\sqrt{2n}}\right)
		\frac{1}{2}
		\frac{1}{-\sqrt{2n}-r}
		\right]\\
		&=
		\frac{n}{2n-r^{2}}
		\left[
		\left(\frac{n-1}{r^{2}}+\frac{1}{2}\right)
		+
		\left(\frac{r}{2}-\frac{n-1}{r}+\frac{1}{\sqrt{2n}}\right)
		\frac{1}{-\sqrt{2n}-r}
		\right]\\
		&=
		\frac{n}{2n-r^{2}}
		\frac{1}{r^{2}(r+\sqrt{2n})}
		\bigg[
		(n-1)r+\sqrt{2n}(n-1)
		+\frac{\sqrt{2n}}{2}r^{2}\\
		&\qquad\qquad
		+(n-1)r
		-\frac{1}{\sqrt{2n}}r^{2}
		\bigg]\\
			\end{aligned}
	\end{equation*}
	\begin{equation*}
		\begin{aligned}
		&=
		\frac{n}{2n-r^{2}}
		\frac{1}{r^{2}(r+\sqrt{2n})}
		\left[
		\frac{1}{\sqrt{2n}}(n-1)r^{2}
		+2(n-1)r
		+\sqrt{2n}(n-1)
		\right]\\
		&=
		\frac{n}{2n-r^{2}}
		\frac{1}{r^{2}(r+\sqrt{2n})}
		\frac{n-1}{\sqrt{2n}}
		(r+\sqrt{2n})^{2}\\
		&=
		\frac{n(n-1)}{\sqrt{2n}}
		\frac{1}{r^2}
		\frac{1}{\sqrt{2n}-r}\\
		&>
		\frac{n-1}{8n}
		>0.
	\end{aligned}
\end{equation*}
Thus $\eta>0$. The lemma follows.
\end{proof}

Lemma~\ref{lem:sphere-comparison} can be viewed as another version of the comparison theorem near the sphere given in Appendix B of \cite{Drugan2013ImmersedS}.

\subsection{Construction of the Angenent Torus and Classification of Solutions}

\begin{figure}[h]
	\centering
	\makebox[\textwidth][c]{\includegraphics[width=5.5in]{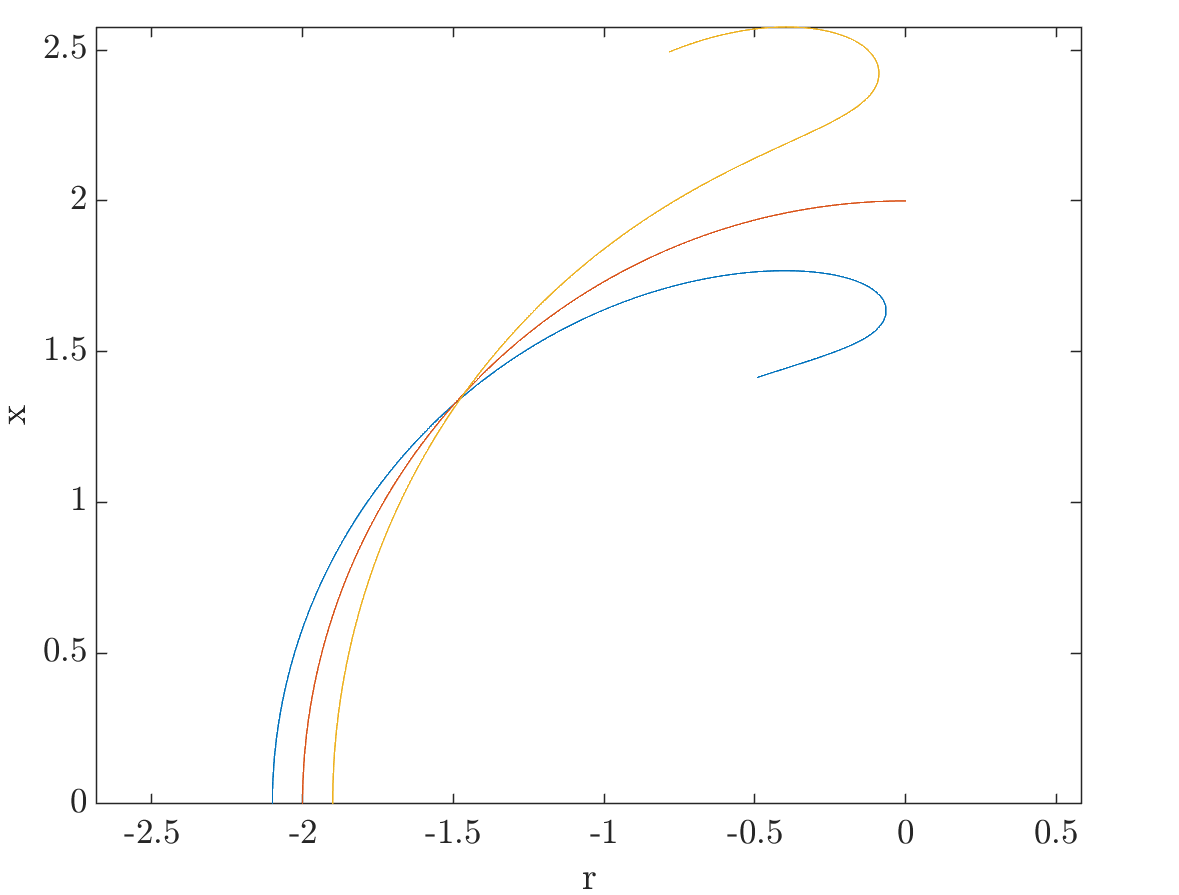}}
	\caption{}
	\label{fig:0_pic}
\end{figure}

This subsection studies how the solutions of the initial value problem \eqref{ivp:uR} vary with the initial value $u(0)=R$. We prove the existence of the Angenent torus by studying the intersection of the trajectories formed by the vertical points of the solutions $u_R$. The following theorem gives a new proof of results analogous to Lemma 9 and Lemma 10 in \cite{Drugan2013ImmersedS}. The key point of the proof is the behavior of solutions near the sphere. Figure~\ref{fig:0_pic} illustrates the typical shapes of these solutions. When the initial value satisfies $R>\sqrt{2n}$, the solution $u_R$ has a vertical point at which $u_R'=-\infty$; while when $R<\sqrt{2(n-1)}$, the solution $u_R$ has a vertical point at which $u_R'=+\infty$.

\begin{theorem}\label{thm:sphere-near-solutions}
(\cite{Drugan2013ImmersedS})
({\romannumeral1}): When $R\in(\sqrt{2n},+\infty)$, the solution $u_R$ of the initial value problem \eqref{ivp:uR} has a vertical point at which $u_R'=-\infty$. Moreover, $u_R''<0$ on the interval from the initial point to this vertical point, and the vertical points form a continuous curve in the $(x,r)$-plane. This curve starts from $(0,+\infty)$, extends to $(\sqrt{2n},0)$, and is bounded in the $x$-direction.

({\romannumeral2}): When $R\in(\sqrt{2(n-1)},\sqrt{2n})$, for the solution $u_R$ of the initial value problem \eqref{ivp:uR}, there exists a corresponding value of $x$ such that $u_R''(x)=0$, and $u_R$ always has a horizontal point which is a minimum point.

({\romannumeral3}): When $R\in(0,\sqrt{2(n-1)})$, the solution $u_R$ of the initial value problem \eqref{ivp:uR} has a vertical point at which $u_R'=+\infty$. Moreover, $u_R''>0$ on the interval from the initial point to this vertical point, and the vertical points form a continuous curve which starts from $(0,0)$ and extends outside every bounded set.
\end{theorem}

\begin{proof}
We first prove the following fact: if a solution $g$ of \eqref{eq:g-equation} is defined at some $r<0$ and also at $r=0$, then one must have $g'(0)=0$.
We prove this by contradiction. Suppose that $g'(0)>0$. Then on some small interval $(a,0)\subset (-\infty,0)$, there exists a constant $C_4>0$, depending on $g$, such that $g'>C_4$. By equation \eqref{eq:g-equation}, we have
\begin{equation*}
	\begin{aligned}
		\frac{\cos\varphi}{\sin\varphi}\frac{d\varphi}{dr}
		&=\frac{r}{2}-\frac{n-1}{r}-\frac{g}{2g'}\\
		&>\frac{r}{2}-\frac{n-1}{r}-\frac{g}{2C_4}\\
		&>\frac{r}{2}-\frac{n-1}{r}-C_9,
	\end{aligned}
\end{equation*}
where $C_9>0$ is a constant depending on the values of $g$ on $(a,0)$ and on $C_4$. Integrating the above inequality over $(a,r)$, we obtain
\begin{equation*}
	\sin\varphi(r)>
	\sin\varphi(a)
	e^{\frac{1}{4}(r^2-a^2)-C_9(r-a)}
	\left(\frac{a}{r}\right)^{n-1}.
\end{equation*}
However, the right-hand side tends to $+\infty$ as $r\to0$, which is a contradiction. A similar argument rules out the case $g'(0)<0$. Therefore, $g'(0)=0$.

For convenience, we first discuss the case on the negative half of the $r$-axis, and then obtain the corresponding conclusion on the positive half by symmetry.

({\romannumeral1}): Suppose that $R\in(-\infty,-\sqrt{2n})$ is close to $-\sqrt{2n}$. Let $(R,D_R)$ be the maximal interval of existence of the corresponding solution $g_R$. If $D_R=0$, then by the fact proved above, we have $g_R'(0)=0$. This means that the graph of $g_R$ generates a convex self-shrinker homeomorphic to a sphere. By the theorem of Huisken \cite{Huisken1990AsymptoticbehaviorFS}, a compact mean-convex self-shrinker must be a sphere. However, the spherical solution corresponds to the initial point $R=-\sqrt{2n}$, which contradicts the assumption $R<-\sqrt{2n}$. Hence $D_R<0$. By Lemma~\ref{lem:sphere-comparison}, we have $\varphi_R<\varphi_{-\sqrt{2n}}$. If $g_R'(D_R)=+\infty$, then this contradicts $\varphi_R<\varphi_{-\sqrt{2n}}$. Therefore, $g_R'(D_R)=-\infty$. Since $g_R'(R)=+\infty$, by continuity there exists a point in $(R,D_R)$ at which $g_R'=0$. Thus, when $R\in(-\infty,-\sqrt{2n})$ is close to $-\sqrt{2n}$, the solution $g_R$ must have a horizontal point. Moreover, by the continuous dependence of solutions on initial data, these horizontal points converge to $(x,r)=(\sqrt{2n},0)$.

We now discuss the corresponding situation on the positive half of the $r$-axis. By the conclusion proved in the preceding paragraph, when $R\in(\sqrt{2n},+\infty)$ is close to $\sqrt{2n}$, the solution $u_R$ has a vertical point. By Lemma~\ref{lem:vertical-point-curve}, Lemma~\ref{lem:vertical-point-continuity}, and Lemma~\ref{lem:large-R-vertical-point}, the vertical points form a continuous curve starting from $(0,+\infty)$. The boundedness of this curve in the $x$-direction is contained in Lemma~\ref{lem:concavity-propagation}. Hence this curve extends from $(0,+\infty)$ to $(\sqrt{2n},0)$.

({\romannumeral2}): Suppose that $R\in(-\sqrt{2n},-\sqrt{2(n-1)})$ is close to $-\sqrt{2n}$. Let $(R,D_R)$ be the maximal interval of existence of the corresponding solution $g_R$. By Lemma~\ref{lem:sphere-comparison}, we have $\varphi_R>\varphi_{-\sqrt{2n}}$. Using an argument similar to that in ({\romannumeral1}), we obtain $D_R<0$ and $g_R'(D_R)=+\infty$. Thus, $\lim_{r\to D_R^-}g_R''(r)=+\infty$. On the other hand, near the initial point $R$, the solution is close to the spherical solution, and hence $g_R''<0$ there. Therefore $g_R''$ must have a zero in the interior of the interval. Moreover, since $g_R'(D_R)=+\infty$, the corresponding solution $u_R$ has a horizontal point. At an interior zero of $g_R''$ where $g_R'$ is finite and nonzero, the relation $u_R''=-\frac{g_R''}{(g_R')^3}$ implies that $u_R''$ also has a zero. By Lemma~\ref{lem:maximum-principle}, the horizontal point of $u_R$ must be a minimum point. Therefore, by an argument analogous to the previous discussion on the continuity of the horizontal points of $g$, for $R\in(\sqrt{2(n-1)},\sqrt{2n})$, the horizontal point of $u_R$ persists continuously, and a zero of $u_R''$ always exists.

({\romannumeral3}): By Lemma~\ref{lem:small-R-vertical-point}, when $R$ is sufficiently small, the solution $u_R$ of \eqref{eq:u-equation} has a vertical point, and $u_R''>0$. By the continuity of the horizontal points of $g_R$, the solution $u_R$ always has this property, unless $u_R$ extends outside every bounded set. This would mean that for some $R\in(0,\sqrt{2(n-1)})$, the solution $u_R$ is unbounded. By the result on rotationally symmetric self-shrinkers in \cite{Kleene2010SelfshrinkersWA}, the only self-shrinkers which extend to infinity in the $x$-direction are trumpet-shaped self-shrinkers and the cylinder $r=\sqrt{2(n-1)}$. The trumpet-shaped self-shrinkers must be non-embedded, whereas the solutions considered here are symmetric and hence embedded. Therefore, the trumpet-shaped case cannot occur. By Lemma 4 in \cite{Drugan2013ImmersedS}, if $b_R<+\infty$, then $\lim_{x\to b_R}u_R(x)<+\infty$. Therefore, for every $R\in(0,\sqrt{2(n-1)})$, the vertical point of $u_R$ is bounded. As $R\to\sqrt{2(n-1)}$, the continuous dependence of solutions on initial data implies that $u_R$ converges to the straight line $r=\sqrt{2(n-1)}$. In this limit, the vertical point escapes every bounded set. Combining this with Lemma~\ref{lem:small-R-vertical-point}, we obtain ({\romannumeral3}).
\end{proof}

\begin{theorem}\label{thm:angenent-torus-existence}
(\cite{Angenent1992})
There exists an embedded rotationally symmetric self-shrinking torus, namely the Angenent torus.
\end{theorem}

\begin{proof}
By Theorem~\ref{thm:sphere-near-solutions}, as $R$ varies from $+\infty$ to $\sqrt{2n}$, the vertical points of the solutions $u_R$ of the initial value problem \eqref{ivp:uR} form a continuous curve. This curve starts from $(0,+\infty)$, extends to $(\sqrt{2n},0)$, and is bounded in the $x$-direction. Similarly, as $R$ varies from $0$ to $\sqrt{2(n-1)}$, the vertical points of the solutions $u_R$ form a continuous curve starting from $(0,0)$ and escaping every bounded set.

The first curve, together with suitable portions of the coordinate axes, bounds a region in the $(x,r)$-plane. Since the second curve starts from the boundary of this region and eventually leaves every bounded set, by continuity it must intersect the first curve.

At an intersection point, the two solution curves have vertical tangent lines. Equivalently, when they are written as graphs $x=g(r)$, they have the same initial data $g(r_*)=x_*$ and $g'(r_*)=0$. Since equation \eqref{eq:g-equation} is a regular ordinary differential equation at $g'=0$, the uniqueness theorem for ordinary differential equations implies that the two curves glue together smoothly at the intersection point. Thus the curve $\Gamma^{+}$ formed by the two solution curves is smooth. Its reflected curve $\Gamma^{-}$ with respect to the $r$-axis also exists, and $\Gamma^{+}\cup\Gamma^{-}$ is a smooth curve homeomorphic to $\mathbb{S}^1$. Rotating this curve about the $x$-axis gives an embedded torus.
\end{proof}

\section{Application of the Comparison Theorem \expandafter{\romannumeral 2}: Monotonicity of Horizontal Points}

\begin{figure}[h]
	\centering
	\makebox[\textwidth][c]{\includegraphics[width=6in]{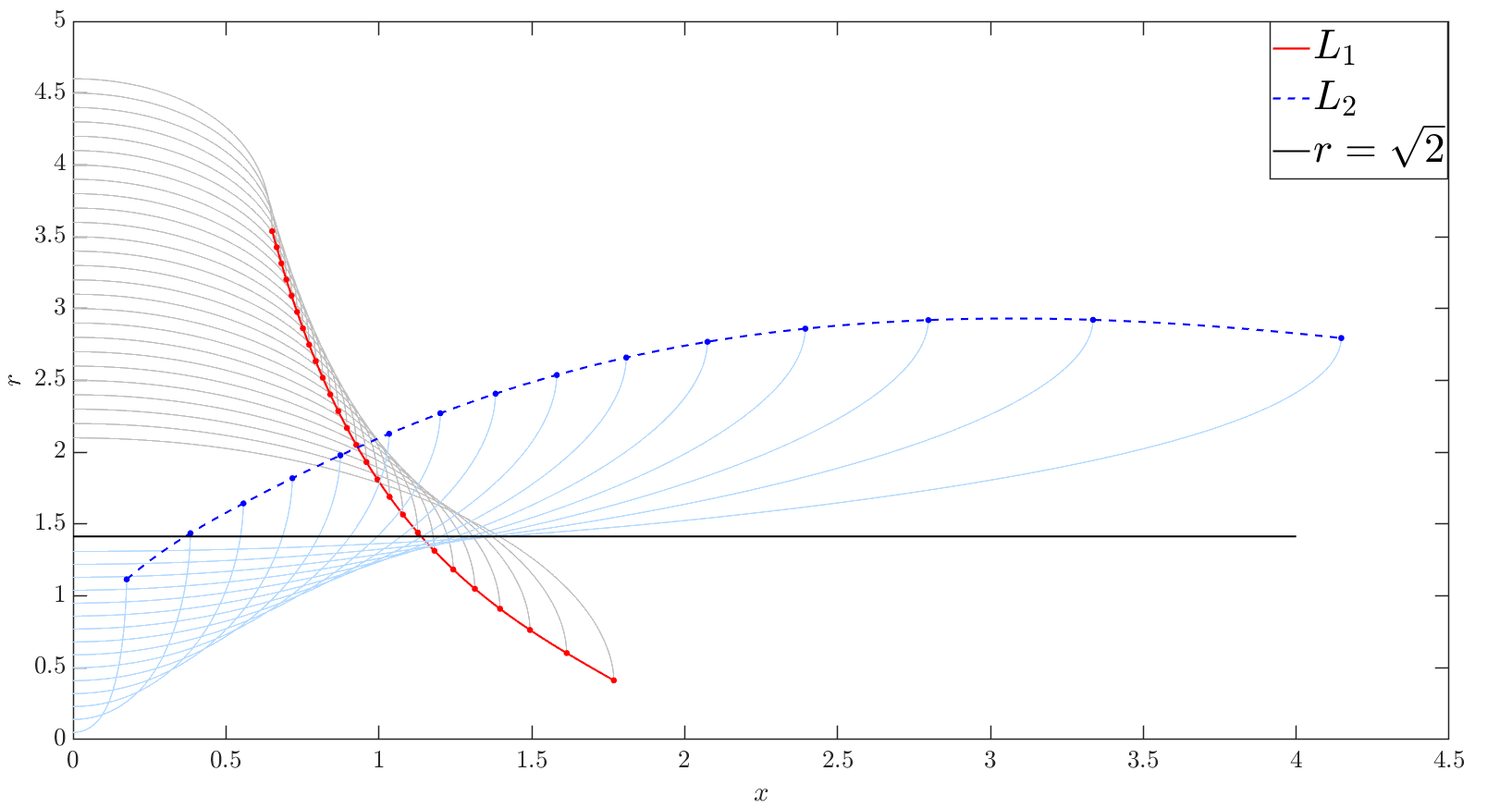}}
	\caption{}
	\label{fig:1_pic}
\end{figure}

After the construction of the Angenent torus, a natural question is the uniqueness of such a torus. In this section, we consider this problem from the viewpoint of the trajectories of vertical points, and prove some partial results by using the comparison theorem.

We first introduce the following definition of monotonicity.

\begin{defn}
Suppose that $R\in(-\infty,-\sqrt{2n})$, and denote the horizontal point of $g_R$ by $(\bar R,g_R(\bar R))$. If there exists an interval $I_R\subset(-\infty,-\sqrt{2n})$ containing $R$, such that for any $R_1,R_2\in I_R$,
\begin{equation*}
	R_2>R_1
	\Rightarrow
	\bar R_2>\bar R_1,\qquad
	g_{R_2}(\bar R_2)>g_{R_1}(\bar R_1),
\end{equation*}
then we say that the curve of horizontal points of the solutions $g_{\tilde R}$ is monotone increasing for $\tilde R\in I_R$, or that the curve of horizontal points of $g_R$ is monotone increasing at $R$.

We denote by $L_1$ the curve formed by the horizontal points of the solutions $g_R$. The above monotonicity means that $L_1$ is the graph of a monotone increasing function; in this case, we say that $L_1$ is monotone increasing. Sometimes we also consider the reflection of $L_1$ with respect to the $r$-axis, still denoted by $L_1$. This corresponds to the trajectory of vertical points of the solutions $u_R$ for $R\in(\sqrt{2n},+\infty)$. Under this reflection, the above definition means that $L_1$ is the graph of a decreasing function. If neither of these cases occurs, then we say that $L_1$ is non-monotone.

When $R\in(0,\sqrt{2(n-1)})$, we denote by $L_2$ the curve formed by the horizontal points of the corresponding solutions $g_R$. A similar definition of monotonicity applies to $L_2$. The curves $L_1$ and $L_2$ defined in this way depend on the initial slope $C$, but we will usually not indicate this dependence in the notation.
\end{defn}

In the discussion of the uniqueness of the Angenent torus, we use the trajectories of vertical points as the starting point. The previous section shows that the existence of the Angenent torus is related to the existence of an intersection between $L_1$ and $L_2$ in the case $C=+\infty$. Since the previous section established the existence of $L_1$ and $L_2$, it is natural to further study their properties. The following theorem shows that the monotonicity of $L_1$ or $L_2$ is related to the conditions $h>0$ and $\eta\geq0$.

\begin{theorem}\label{thm:horizontal-point-monotonicity}
Let $C>0$ be fixed. In this theorem, we write $g_R$ for $g_{R,C}$ for simplicity. Suppose that for some $R\in(-\infty,-\sqrt{2n})$, the solution $g_R$ satisfies
\begin{equation*}
	\eta
	=
	(1+(g'_{R})^{2})
	\left(\frac{n-1}{r^{2}}+\frac{1}{2}\right)
	+
	2\left[
	\left(\frac{r}{2}-\frac{n-1}{r}\right)
	-
	\left(\frac{R}{2}-\frac{n-1}{R}\right)
	\right]g'_{R}g''_{R}
	\geq0
\end{equation*}
on the portion from the initial point to the horizontal point, namely on the portion where $g_R\geq0$ and $g_R'\geq0$. Then the curve $L_1$ of horizontal points of the solutions $g_R$ is monotone increasing at $R$.

For $R\in(0,\sqrt{2(n-1)})$, if the corresponding solution $g_R$ satisfies the same inequality on the corresponding portion, then the same conclusion holds for $L_2$.

Moreover, if for all sufficiently large $C>0$, the solutions $g_{R,C}$ satisfy the strict condition $\eta>0$ on the corresponding portion, then the same monotonicity conclusion holds in the limiting case $C=+\infty$.
\end{theorem}

\begin{proof}
The following argument is carried out in the region where $r=u(x)<0$, in the sense of the formal extension of the equation to the negative $r$-axis. The corresponding conclusion on the positive side follows by symmetry.

Consider the solutions $u_1$ and $u_2$ of \eqref{eq:u-equation} corresponding respectively to $g_{R_1}$ and $g_{R_2}$. Their initial values satisfy $u_1(0)<u_2(0)<-\sqrt{2n}$ and $u_i'(0)=\frac{1}{C}$, $i=1,2$. Suppose that both solutions have vertical points, denoted by $(x_1,u_1(x_1))$ and $(x_2,u_2(x_2))$. In what follows, we only consider the interval from $x=0$ to the corresponding vertical point, so $x\geq0$.

We first consider the case where the two solution curves do not intersect. Then $u_1<u_2$ on their common interval of definition. We claim that $\theta_1\geq\theta_2$ on the common interval. Indeed, $\theta_1(0)=\theta_2(0)$ and $\theta_1'(0)>\theta_2'(0)$. If the claim were false, then there would exist a first point $\bar x>0$ such that $\theta_1(\bar x)=\theta_2(\bar x)$, while $\theta_1>\theta_2$ on $(0,\bar x)$. Since $u_i'=\tan\theta_i$, we have $u_1'(\bar x)=u_2'(\bar x)$.

On the other hand, the function $r\mapsto \frac{n-1}{r}-\frac{r}{2}$ is decreasing for $r<0$. Therefore, since $u_1<u_2$, we have on the common interval
\begin{equation}\label{eq:r-function-monotonicity}
	\frac{n-1}{u_1}-\frac{u_1}{2}
	>
	\frac{n-1}{u_2}-\frac{u_2}{2}.
\end{equation}
Thus, at $\bar x$,
\begin{equation*}
	\frac{d\theta_1}{dx}
	-
	\frac{d\theta_2}{dx}
	=
	\frac{n-1}{u_1}-\frac{u_1}{2}
	-
	\left[
	\frac{n-1}{u_2}-\frac{u_2}{2}
	\right]
	>0.
\end{equation*}
This contradicts the choice of $\bar x$ as the first contact point. Hence $\theta_1\geq\theta_2$.

Before the vertical points, the angles lie in an interval where $\tan$ is monotone. Hence $u_1'\geq u_2'$. Using \eqref{eq:u-equation}, \eqref{eq:r-function-monotonicity}, and $x\geq0$, we obtain
\begin{equation*}
	\begin{aligned}
		\frac{d\theta_1}{dx}
		&=
		\frac{n-1}{u_1}
		-\frac{u_1}{2}
		+\frac{x}{2}u_1'\\
		&>
		\frac{n-1}{u_2}
		-\frac{u_2}{2}
		+\frac{x}{2}u_2'
		=
		\frac{d\theta_2}{dx}.
	\end{aligned}
\end{equation*}
Therefore, $\theta_1$ reaches $\frac{\pi}{2}$ faster than $\theta_2$. Hence $x_1<x_2$. Since the curves of $u_1$ and $u_2$ do not intersect and $u_1',u_2'>0$, we also have $u_1(x_1)<u_2(x_2)$. This proves the desired monotonicity in the non-intersecting case.

We now assume that the two curves $u_1$ and $u_2$ intersect. Consider the solution $g_{R_1}$ corresponding to $u_1$, and let $h_{R_1}$ be the corresponding linearized function. If $g_{R_1}$ satisfies the condition $\eta\geq0$, then Lemma~\ref{lem:eta-comparison} implies that $h_{R_1}>0$. Therefore, when $R_2>R_1$ and $R_2$ is sufficiently close to $R_1$, we have $\varphi_{R_2}>\varphi_{R_1}$, or equivalently, $g'_{R_2}>g'_{R_1}$ on the common interval under consideration. Hence $g_{R_2}-g_{R_1}$ is strictly increasing on this interval. In particular, the curves of $g_{R_2}$ and $g_{R_1}$ can intersect at most once.

Let their horizontal points be $(\bar R_2,g_{R_2}(\bar R_2))$ and $(\bar R_1,g_{R_1}(\bar R_1))$. We first prove by contradiction that $\bar R_2>\bar R_1$. Suppose instead that $\bar R_2\leq\bar R_1$. By the concavity property established previously, we have $g''_{R_1}<0$ on the portion from the initial point to the horizontal point. Hence
\begin{equation}\label{eq:horizontal-point-derivative-order}
	g'_{R_1}(\bar R_2)\geq g'_{R_1}(\bar R_1)=0.
\end{equation}
On the other hand, $g'_{R_2}(\bar R_2)=0=g'_{R_1}(\bar R_1)$. Combining this with \eqref{eq:horizontal-point-derivative-order}, we obtain $g'_{R_2}(\bar R_2)\leq g'_{R_1}(\bar R_2)$, which contradicts $g'_{R_2}>g'_{R_1}$. Therefore, $\bar R_2>\bar R_1$.

Next we prove that $g_{R_2}(\bar R_2)>g_{R_1}(\bar R_1)$. Suppose, to the contrary, that $g_{R_1}(\bar R_1)\geq g_{R_2}(\bar R_2)$. Since we are now considering the case where the two curves intersect, and since the curves of $g_{R_2}$ and $g_{R_1}$ can intersect at most once, after their intersection we have $g_{R_2}>g_{R_1}$. In particular, at $r=\bar R_1$, $g_{R_2}(\bar R_1)>g_{R_1}(\bar R_1)$. Combining this with the contrary assumption gives $g_{R_2}(\bar R_1)>g_{R_2}(\bar R_2)$. But we have already proved that $\bar R_1<\bar R_2$, and this contradicts $g'_{R_2}>0$ on the interval before the horizontal point. Therefore, $g_{R_2}(\bar R_2)>g_{R_1}(\bar R_1)$. This proves the monotonicity near $R_1$.

The proof for $L_2$ is analogous.

Finally, suppose that for all sufficiently large $C>0$, the solutions $g_{R,C}$ satisfy the strict condition $\eta>0$. By Lemma~\ref{lem:eta-comparison}, the limiting linearized function also satisfies $h_{R,+\infty}>0$. Therefore the same argument as above applies to the limiting solutions with $C=+\infty$, and the corresponding monotonicity conclusion follows.
\end{proof}

For the monotonicity condition given in the theorem above, we have already verified that it holds for the special solution $x^{2}+r^{2}=2n$. Moreover, according to some numerical computations, as $R$ varies from $+\infty$ to $\sqrt{2n}$, the horizontal points of $g_R$ are monotone decreasing. When $R$ varies from $0$ to $\sqrt{2(n-1)}$, the horizontal points of $g_R$ are monotone increasing as long as their $r$-coordinates are less than $\sqrt{2(n-1)}$. The following theorem gives a partial result concerning these monotonicity properties.

\begin{figure}[h]
	\centering
	\makebox[\textwidth][c]{\includegraphics[width=6in]{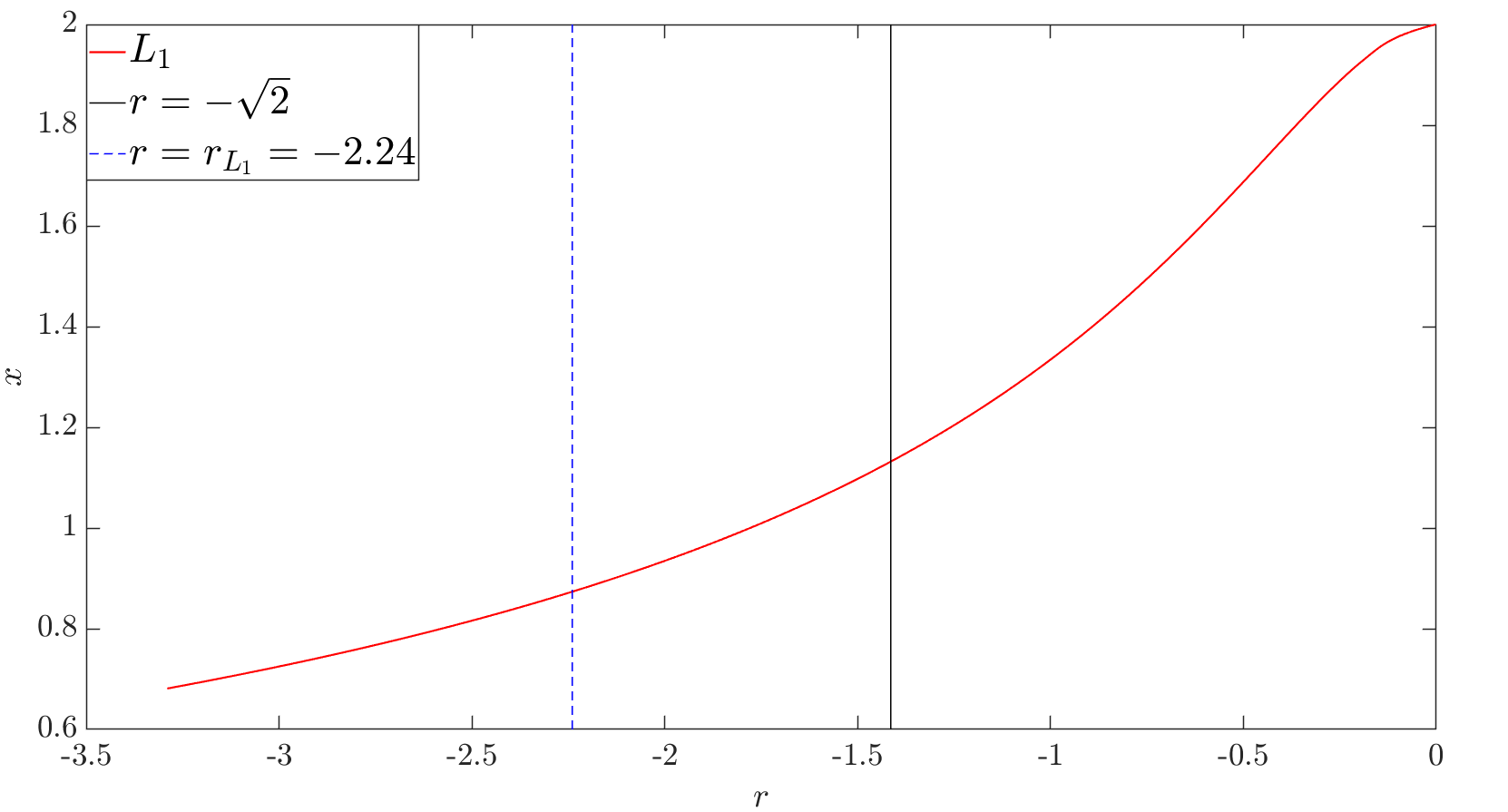}}
	\caption{}
	\label{fig:2_pic}
\end{figure}

\begin{figure}[h]
	\centering
	\makebox[\textwidth][c]{\includegraphics[width=6in]{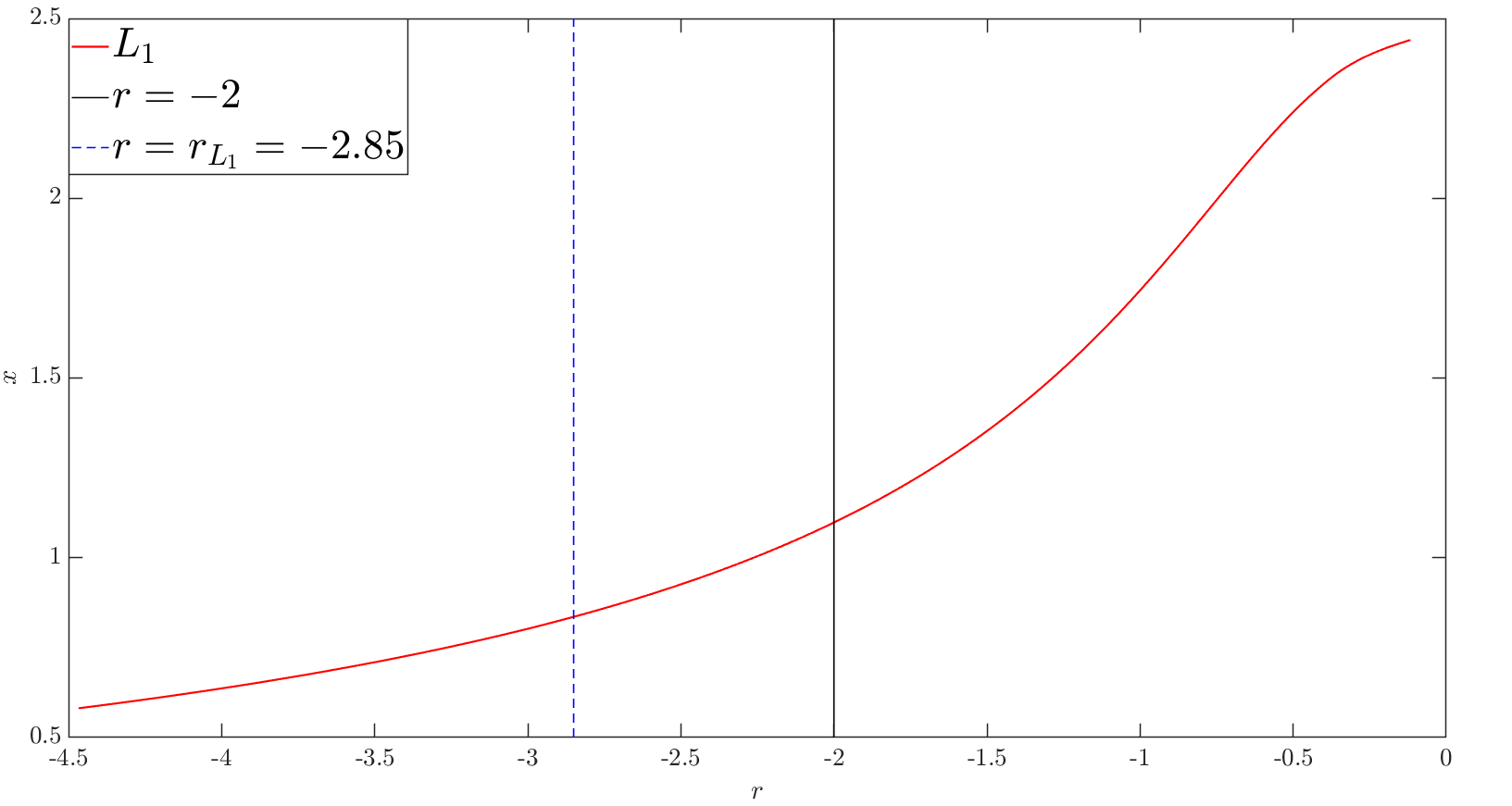}}
	\caption{}
	\label{fig:3_pic}
\end{figure}

\begin{figure}[h]
	\centering
	\makebox[\textwidth][c]{\includegraphics[width=6in]{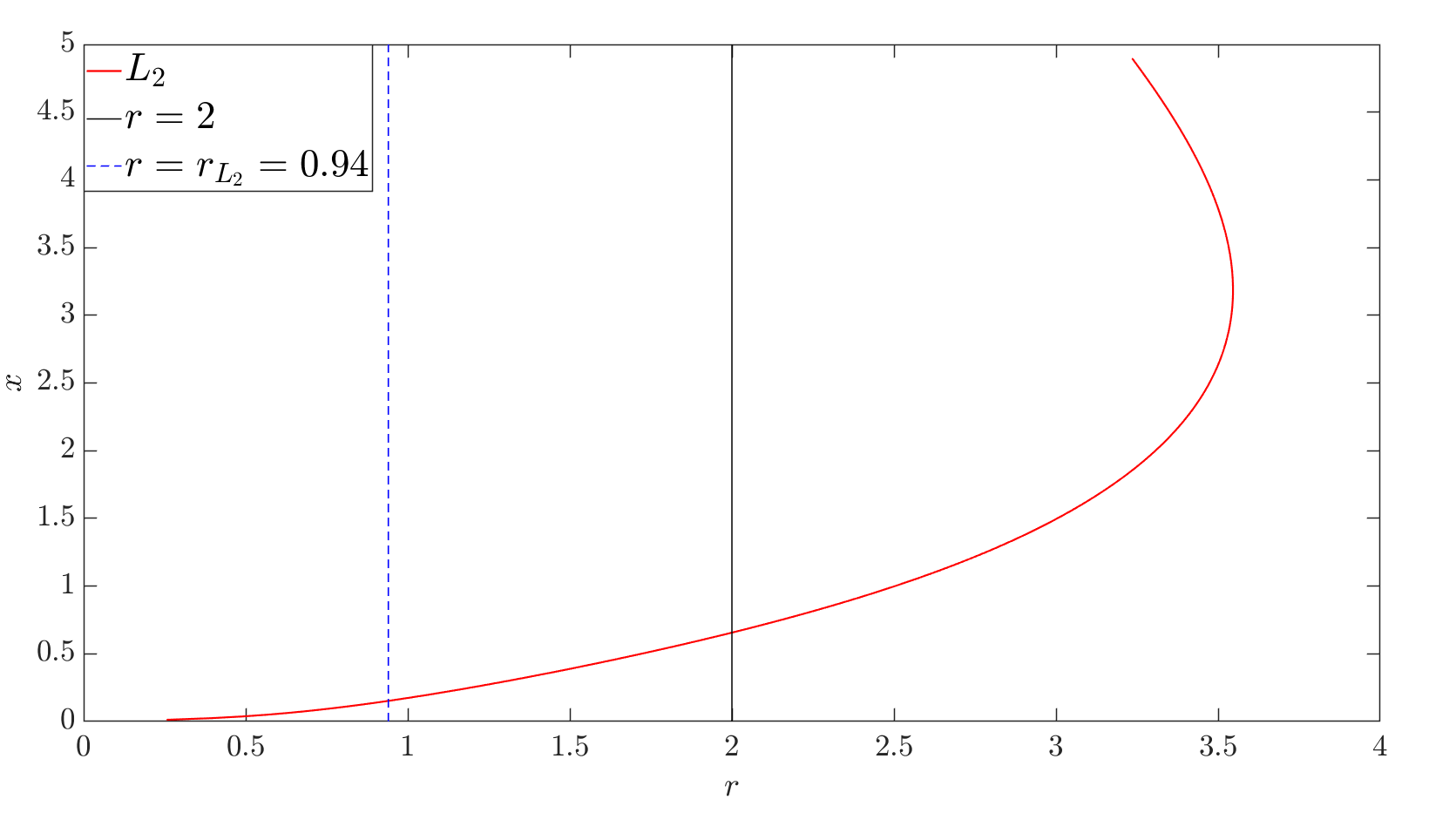}}
	\caption{}
	\label{fig:4_pic}
\end{figure}

The main idea of the proof may be summarized as a successive differentiation argument. Namely, we differentiate the inequality that we want to prove, and ensure that the original inequality can be recovered by integrating the differentiated inequality. After differentiating several times, we eventually obtain an inequality depending only on $r$, which determines the range of $r$ for which the curve of horizontal points remains monotone.

Figure~\ref{fig:2_pic} illustrates the curve $L_1$ and $r=r_{L_1}$  for $n=2$.
Figure~\ref{fig:3_pic} and Figure~\ref{fig:4_pic} illustrate the curve $L_1$, $L_2$, $r=r_{L_1}$ and $r=r_{L_2}$ for $n=3$.

\begin{theorem}\label{thm:monotonicity-range}
Suppose that $R\in(-\infty,-\sqrt{2n})$. Then the curve $L_1$ of horizontal points of the solutions $g_R$ is monotone increasing when $r<r_{L_1}$, 	where $r_{L_1}<-\sqrt{2(n-1)}$ is given by
\[
r_{L_1}=-\sqrt{2(n-1)+1+\sqrt{8(n-1)+1}}.
\]
Moreover, $L_1$ intersects the line $r=r_{L_1}$ at exactly one point. The same conclusion also holds when $C=+\infty$.

For $n\geq3$, when $R\in(0,\sqrt{2(n-1)})$, the curve $L_2$ of horizontal points of the solutions $g_R$ is monotone increasing when $0<r<r_{L_2}$. Here $0<r_{L_2}<\sqrt{2(n-1)}$ is defined by $r_{L_2}=\min\{r_1,r_2,r_3\}$, where
\[
r_1=\sqrt{2\left[n-\frac{1}{2}-\sqrt{2n-\frac{7}{4}}\right]},
\]
\begin{equation*}
	r_{2}=\sqrt{
		2\left[
		\frac{(n-1)(4n-5)}{2(2n-3)}
		-(n-1)\sqrt{\left(\frac{4n-5}{4n-6}\right)^2-1}
		\right]},
\end{equation*}
and $$r_3=\sqrt{2(n-1)\frac{n-2}{n}}.$$
Moreover, $L_2$ intersects the line $r=r_{L_2}$ at exactly one point. The same conclusion also holds when $C=+\infty$.
\end{theorem}

\begin{proof}
To prove $\eta\geq0$, it is enough to prove
\begin{equation}\label{eq:eta-bracket-nonpositive}
	\frac{1+(g')^{2}}{g'g''}
	\left(\frac{n-1}{r^{2}}+\frac{1}{2}\right)
	+
	2\left[
	\left(\frac{r}{2}-\frac{n-1}{r}\right)
	-
	\left(\frac{R}{2}-\frac{n-1}{R}\right)
	\right]
	\leq 0.
\end{equation}
Indeed, on the portion under consideration we have $g'>0$ and $g''<0$, and hence $g'g''<0$. Therefore, multiplying \eqref{eq:eta-bracket-nonpositive} by $g'g''$ gives $\eta\geq0$.

If we prove
\begin{equation}\label{eq:eta-bracket-derivative-nonpositive}
	\frac{d}{dr}
	\bigg\{
	\frac{1+(g')^{2}}{g'g''}
	\left(\frac{n-1}{r^{2}}+\frac{1}{2}\right)
	+
	2\left[
	\left(\frac{r}{2}-\frac{n-1}{r}\right)
	-
	\left(\frac{R}{2}-\frac{n-1}{R}\right)
	\right]
	\bigg\}
	\leq0,
\end{equation}
then integrating the above inequality over the interval $(R,r)$, and using $g''(R)<0$, we obtain
\begin{equation}\label{eq:eta-bracket-initial-bound}
	\begin{aligned}
		&
		\frac{1+(g')^{2}}{g'g''}
		\left(\frac{n-1}{r^{2}}+\frac{1}{2}\right)
		+
		2\left[
		\left(\frac{r}{2}-\frac{n-1}{r}\right)
		-
		\left(\frac{R}{2}-\frac{n-1}{R}\right)
		\right] \\
		&\leq
		\frac{1+(g')^{2}}{g'g''}(R)
		\left(\frac{n-1}{R^{2}}+\frac{1}{2}\right)<0
	\end{aligned}
\end{equation}
Thus \eqref{eq:eta-bracket-nonpositive} follows. Hence it remains to prove \eqref{eq:eta-bracket-derivative-nonpositive}. Namely, we need to show that
\begin{equation}\label{eq:eta-bracket-derivative-expanded}
	\begin{aligned}
		&\frac{d}{dr}
		\bigg\{
		\frac{1+(g')^{2}}{g'g''}
		\left(\frac{n-1}{r^{2}}+\frac{1}{2}\right)
		+
		2\left[
		\left(\frac{r}{2}-\frac{n-1}{r}\right)
		-
		\left(\frac{R}{2}-\frac{n-1}{R}\right)
		\right]
		\bigg\} \\
		&=
		\frac{1}{(g')^2(g'')^2}
		\bigg\{
		2(g')^2(g'')^2
		-
		(1+(g')^{2})\left[(g'')^2+g'g'''\right]
		\bigg\}
		\left(\frac{n-1}{r^2}+\frac{1}{2}\right)\\
		&\quad
		+
		\frac{1+(g')^{2}}{g'g''}
		\left[-\frac{2(n-1)}{r^3}\right]
		+
		2\left[\frac{1}{2}+\frac{n-1}{r^2}\right]\\
		&=
		4\left(\frac{n-1}{r^2}+\frac{1}{2}\right)
		-
		\left(\frac{n-1}{r^2}+\frac{1}{2}\right)
		\left[
		\frac{1}{(g')^2}
		+
		1
		+
		\frac{(1+(g')^{2})g'''}{g'(g'')^2}
		\right]\\
		&\quad
		-
		\frac{2(n-1)}{r^3}
		\frac{1+(g')^{2}}{g'g''}\\
		&=
		3\left(\frac{n-1}{r^2}+\frac{1}{2}\right)
		-
		\left(\frac{n-1}{r^2}+\frac{1}{2}\right)
		\frac{1}{(g')^2}
		-\left(\frac{n-1}{r^2}+\frac{1}{2}\right)
		\frac{1+(g')^{2}}{g'(g'')^2}g'''\\
		&\quad
		-
		\frac{2(n-1)}{r^3}
		\frac{1+(g')^{2}}{g'g''}
		\leq0.
	\end{aligned}
\end{equation}

For the next step, we multiply both sides of the above inequality by $(g')^2(g'')^2$, which is positive away from possible isolated zeroes of $g''$, and substitute the expression \eqref{eq:g-third-derivative} for $g'''$. After simplification, it is enough to prove
\begin{equation*}
	\begin{aligned}
		&
		3\left(\frac{n-1}{r^2}+\frac{1}{2}\right)(g')^2(g'')^2
		-
		\left(\frac{n-1}{r^2}+\frac{1}{2}\right)(g'')^2\\
		&\quad
		-
		\left(\frac{n-1}{r^2}+\frac{1}{2}\right)
		(1+(g')^{2})g'g'''
		-
		\frac{2(n-1)}{r^3}(1+(g')^{2})g'g''\\
		&=
		-
		\left(\frac{n-1}{r^2}+\frac{1}{2}\right)(g'')^2
		-
		\frac{n-1}{r^2}
		\left(\frac{n-1}{r^2}+\frac{1}{2}\right)
		(1+(g')^{2})^2(g')^{2}\\
		&\quad
		-
		\left(\frac{n-1}{r^2}+\frac{1}{2}\right)
		\left(\frac{r}{2}-\frac{n-1}{r}\right)
		(1+(g')^{2})g'g''\\
		&\quad
		-
		\frac{g}{2}
		\left(\frac{n-1}{r^2}+\frac{1}{2}\right)
		(1+(g')^{2})(g')^2g''
		-
		\frac{2(n-1)}{r^3}(1+(g')^{2})g'g''\leq0.
	\end{aligned}
\end{equation*}
Further simplification reduces the desired inequality to
\begin{equation}\label{eq:eta-reduced-inequality}
	\begin{aligned}
		&
		-2\left(\frac{n-1}{r^2}+\frac{1}{2}\right)
		(1+(g')^{2})^2
		\left(\frac{r}{2}-\frac{n-1}{r}\right)^2(g')^2
		-
		\frac{n-1}{r^2}
		\left(\frac{n-1}{r^2}+\frac{1}{2}\right)
		(1+(g')^{2})^2(g')^{2}\\
		&-
		\frac{1}{2}
		\left(\frac{n-1}{r^2}+\frac{1}{2}\right)
		\left(\frac{r}{2}-\frac{n-1}{r}\right)
		(1+(g')^{2})^2g(g')^3
		-
		\frac{2(n-1)}{r^3}
		\left(\frac{r}{2}-\frac{n-1}{r}\right)
		(1+(g')^{2})^2(g')^2\\
		&\quad
		+
		\frac{1}{4}g^2(g')^2
		\left(\frac{n-1}{r^2}+\frac{1}{2}\right)
		(1+(g')^{2})^2\\
		&\quad
		+
		\frac{3}{2}gg'
		\left(\frac{n-1}{r^2}+\frac{1}{2}\right)
		\left(\frac{r}{2}-\frac{n-1}{r}\right)
		(1+(g')^{2})^2+
		\frac{n-1}{r^3}(1+(g')^{2})^2gg'\\
		&\quad
		-
		\frac{g^2}{4}
		\left(\frac{n-1}{r^2}+\frac{1}{2}\right)
		(1+(g')^{2})^2
		\leq0.
	\end{aligned}
\end{equation}
The last term in the above inequality is non-positive. Therefore, in order to prove \eqref{eq:eta-reduced-inequality}, it suffices to show that the first two lines are non-positive and that the third and fourth lines are non-positive.

Dividing the first two lines by $(1+(g')^{2})^2(g')^2$, we are reduced to proving
\begin{equation}\label{eq:eta-reduced-first-part}
	\begin{aligned}
		&
		-2\left(\frac{n-1}{r^2}+\frac{1}{2}\right)
		\left(\frac{r}{2}-\frac{n-1}{r}\right)^2
		-
		\frac{n-1}{r^2}
		\left(\frac{n-1}{r^2}+\frac{1}{2}\right)\\
		&\quad
		-
		\frac{1}{2}
		\left(\frac{n-1}{r^2}+\frac{1}{2}\right)
		\left(\frac{r}{2}-\frac{n-1}{r}\right)gg'
		-
		\frac{2(n-1)}{r^3}
		\left(\frac{r}{2}-\frac{n-1}{r}\right)
		\leq0.
	\end{aligned}
\end{equation}
Similarly, dividing the third and fourth lines of \eqref{eq:eta-reduced-inequality} by $(1+(g')^{2})^2gg'$, we are reduced to proving
\begin{equation}\label{eq:eta-reduced-second-part}
	\begin{aligned}
		&
		\frac{1}{4}gg'
		\left(\frac{n-1}{r^2}+\frac{1}{2}\right)
		+
		\frac{3}{2}
		\left(\frac{n-1}{r^2}+\frac{1}{2}\right)
		\left(\frac{r}{2}-\frac{n-1}{r}\right)
		+
		\frac{n-1}{r^3}
		\leq0.
	\end{aligned}
\end{equation}

We now discuss the intervals on which \eqref{eq:eta-reduced-first-part} and \eqref{eq:eta-reduced-second-part} hold in two separate cases: $r<-\sqrt{2(n-1)}$ and $0<r<\sqrt{2(n-1)}$.

({\romannumeral1}): The case $r<-\sqrt{2(n-1)}$.

Since $r<-\sqrt{2(n-1)}<0$, we have
\[
-\frac{2(n-1)}{r^3}
\left(\frac r2-\frac{n-1}{r}\right)<0.
\]
Moreover,
\[
-\frac{n-1}{r^2}
\left(\frac{n-1}{r^2}+\frac12\right)<0.
\]
Therefore, in order to prove \eqref{eq:eta-reduced-first-part}, it suffices to prove
\begin{equation}\label{eq:L1-key-inequality}
	\left(\frac r2-\frac{n-1}{r}\right)+\frac14 gg'\leq0.
\end{equation}
Also, since $\frac{n-1}{r^3}<0$, in order to prove \eqref{eq:eta-reduced-second-part}, it suffices to prove
\begin{equation*}
	\frac32\left(\frac r2-\frac{n-1}{r}\right)+\frac14 gg'\leq0.
\end{equation*}
Since
\[
\frac32\left(\frac r2-\frac{n-1}{r}\right)
<
\left(\frac r2-\frac{n-1}{r}\right),
\]
we see that \eqref{eq:L1-key-inequality} implies both \eqref{eq:eta-reduced-first-part} and \eqref{eq:eta-reduced-second-part}.

We now prove \eqref{eq:L1-key-inequality}. If
\begin{equation}\label{eq:L1-auxiliary-derivative}
	\frac{d}{dr}
	\left[
	\left(\frac r2-\frac{n-1}{r}\right)\frac{1}{g'}
	+\frac14 g
	\right]\leq0,
\end{equation}
then integrating \eqref{eq:L1-auxiliary-derivative} over the interval $(R,r)$ gives
\begin{equation*}
	\left(\frac r2-\frac{n-1}{r}\right)\frac{1}{g'}
	+\frac14 g
	\leq
	\left(\frac R2-\frac{n-1}{R}\right)\frac{1}{g'(R)}<0,
\end{equation*}
where we used $g'(R)=C>0$ and $\frac R2-\frac{n-1}{R}<0$. Multiplying the above inequality by $g'>0$, we obtain \eqref{eq:L1-key-inequality}.

Thus it remains to prove \eqref{eq:L1-auxiliary-derivative}. We compute
\begin{equation*}
	\begin{aligned}
		&\frac{d}{dr}
		\left[
		\left(\frac r2-\frac{n-1}{r}\right)\frac{1}{g'}
		+\frac14 g
		\right] \\
		&=
		\left(\frac{n-1}{r^2}+\frac12\right)\frac{1}{g'}
		-
		\left(\frac r2-\frac{n-1}{r}\right)
		\frac{g''}{(g')^2}
		+\frac14 g' .
	\end{aligned}
\end{equation*}
Therefore, after multiplying by $(g')^2>0$, it is enough to show
\begin{equation*}
	\begin{aligned}
		&
		\left(\frac{n-1}{r^2}+\frac12\right)g'
		-
		\left(\frac r2-\frac{n-1}{r}\right)g''
		+\frac14(g')^3 \\
		&=
		\left(\frac{n-1}{r^2}+\frac12\right)g'
		-
		\left(\frac r2-\frac{n-1}{r}\right)^2g'
		-
		\left(\frac r2-\frac{n-1}{r}\right)^2(g')^3 \\
		&\quad
		+\frac14(g')^3
		+
		\frac{g}{2}
		\left(\frac r2-\frac{n-1}{r}\right)(1+(g')^2)
		\leq0.
	\end{aligned}
\end{equation*}
Here we have used equation \eqref{eq:g-equation} to substitute for $g''$.

Since $g\geq0$ and $\frac r2-\frac{n-1}{r}<0$, the last term is non-positive. Hence it is enough to prove
\begin{equation}\label{eq:L1-condition-one}
	\begin{aligned}
		&
		\left(\frac{n-1}{r^2}+\frac12\right)g'
		-
		\left(\frac r2-\frac{n-1}{r}\right)^2g' \\
		&=
		\left[
		\left(\frac{n-1}{r^2}+\frac12\right)
		-
		\left(\frac r2-\frac{n-1}{r}\right)^2
		\right]g'
		\leq0,
	\end{aligned}
\end{equation}
and
\begin{equation}\label{eq:L1-condition-two}
	\begin{aligned}
		&
		-\left(\frac r2-\frac{n-1}{r}\right)^2(g')^3
		+\frac14(g')^3 \\
		&=
		\left[
		-\left(\frac r2-\frac{n-1}{r}\right)^2
		+\frac14
		\right](g')^3
		\leq0.
	\end{aligned}
\end{equation}

Since $g'>0$, conditions \eqref{eq:L1-condition-one} and \eqref{eq:L1-condition-two} are guaranteed by
\begin{equation*}
	\begin{cases}
		\displaystyle
		\left(\frac{n-1}{r^2}+\frac12\right)
		-
		\left(\frac r2-\frac{n-1}{r}\right)^2
		\leq0,\\[1.2em]
		\displaystyle
		-\left(\frac r2-\frac{n-1}{r}\right)^2+\frac14\leq0.
	\end{cases}
\end{equation*}
Solving this system gives
\[
r\leq
-\sqrt{2(n-1)+1+\sqrt{8(n-1)+1}}.
\]
We define
\[
r_{L_1}
=
-\sqrt{2(n-1)+1+\sqrt{8(n-1)+1}}.
\]
Thus $\eta\geq0$ holds in the region $r<r_{L_1}$. By Theorem~\ref{thm:horizontal-point-monotonicity}, the curve $L_1$ of horizontal points of $g_R$ is monotone increasing in this region. Consequently, once $L_1$ crosses the line $r=r_{L_1}$, it cannot return to the region $r<r_{L_1}$. Hence $L_1$ intersects the line $r=r_{L_1}$ at most once; together with the continuity of $L_1$, this gives the desired uniqueness of the intersection.

The case $C=+\infty$ follows by taking the limit $C\to+\infty$.

({\romannumeral2}): The case $0<r<\sqrt{2(n-1)}$.

When $r>0$, the signs of the non-negative terms in \eqref{eq:eta-reduced-first-part} and \eqref{eq:eta-reduced-second-part} are more complicated. We first make some preparations.

It is easy to see that
\[
\frac{2(n-1)}{r^3}
=
\left(\frac{n-1}{r^2}+\frac{1}{2}\right)\frac{2}{r}
-\frac{1}{r},
\]
and
\[
\frac{n-1}{r^2}
=
-\frac{1}{r}
\left(\frac{r}{2}-\frac{n-1}{r}\right)
+\frac{1}{2}.
\]
Substituting these two identities into the left-hand side of \eqref{eq:eta-reduced-first-part}, and then dividing by
\[
\left(\frac{r}{2}-\frac{n-1}{r}\right)
\left(\frac{n-1}{r^2}+\frac{1}{2}\right),
\]
which is negative for $0<r<\sqrt{2(n-1)}$, we see that \eqref{eq:eta-reduced-first-part} is equivalent to the non-negativity of the resulting quotient. This quotient is
\begin{equation*}
	\begin{aligned}
		&-2\left(\frac{r}{2}-\frac{n-1}{r}\right)
		-\frac{1}{r}
		-\frac{1}{2}\frac{1}{\frac{r}{2}-\frac{n-1}{r}}
		+\frac{1}{r}
		\frac{1}{\frac{n-1}{r^2}+\frac{1}{2}}
		-\frac{1}{2}gg'\\
					&=
		-2\left(\frac{r}{2}-\frac{n-1}{r}\right)
		-\frac{1}{r}
		+\frac{1}{\frac{2(n-1)}{r}-r}
		+\frac{1}{r}
		\frac{1}{\frac{n-1}{r^2}+\frac{1}{2}}
		-\frac{1}{2}gg'\\
		&\geq
		-2\left(\frac{r}{2}-\frac{n-1}{r}\right)
		-\frac{1}{r}
		+\frac{r}{2(n-1)}
		+\frac{1}{r}
		\frac{1}{\frac{n-1}{r^2}+\frac{1}{2}}
		-\frac{1}{2}gg'\\
		&=
		\frac{2n-3}{n-1}
		\left(\frac{n-1}{r}-\frac{r}{2}\right)
		+\frac{2r}{2(n-1)+r^2}
		-\frac{1}{2}gg'.
		\end{aligned}
	\end{equation*}
Since $\frac{2r}{2(n-1)+r^2}>0$, it is enough, in order to prove \eqref{eq:eta-reduced-first-part}, to prove
\begin{equation}\label{eq:L2-key-inequality}
\frac{2n-3}{2(n-1)}
\left(\frac{n-1}{r}-\frac{r}{2}\right)
-\frac{1}{4}gg'\geq0.
\end{equation}

Similarly, substituting
\[
\frac{2(n-1)}{r^3}
=
\left(\frac{n-1}{r^2}+\frac{1}{2}\right)\frac{2}{r}
-\frac{1}{r}
\]
into the left-hand side of \eqref{eq:eta-reduced-second-part}, and dividing by $\frac{n-1}{r^2}+\frac{1}{2}$, we obtain
\begin{equation*}
\begin{aligned}
&\frac{1}{4}gg'
+\frac{3}{2}\left(\frac{r}{2}-\frac{n-1}{r}\right)
+\frac{1}{r}
-\frac{1}{2r}
\frac{1}{\frac{n-1}{r^2}+\frac{1}{2}}\\
&=
\frac{1}{4}gg'
+\frac{3}{2}\left(\frac{r}{2}-\frac{n-1}{r}\right)
+\frac{1}{r}
-\frac{r}{2(n-1)+r^2}\\
&=
\frac{1}{4}gg'
+\frac{3}{4}r
-\frac{3n-5}{2r}
-\frac{r}{2(n-1)+r^2}.
\end{aligned}
\end{equation*}
Therefore, to prove \eqref{eq:eta-reduced-second-part}, it suffices to prove
\begin{equation}\label{eq:L2-secondary-inequality}
\frac{3n-5}{2r}
-\frac{3}{4}r
-\frac{1}{4}gg'\geq0.
\end{equation}
A direct computation shows that, when $r\in\left(0,\sqrt{2(n-1)\frac{n-2}{n}}\right]$, we have
\[
\frac{2n-3}{2(n-1)}
\left(\frac{n-1}{r}-\frac{r}{2}\right)
<
\frac{3n-5}{2r}
-\frac{3}{4}r.
\]
Hence, for $n\geq3$, \eqref{eq:L2-key-inequality} implies \eqref{eq:L2-secondary-inequality} in this range.

It remains to determine the range of $r$ for which \eqref{eq:L2-key-inequality} holds. We use the same idea as in \eqref{eq:L1-auxiliary-derivative}. Namely, it is enough to prove
\begin{equation*}
\begin{aligned}
&\frac{d}{dr}
\left\{
\frac{2n-3}{2(n-1)}
\left(\frac{n-1}{r}-\frac{r}{2}\right)
\frac{1}{g'}
-\frac{1}{4}g
\right\}\\
&=
\frac{2n-3}{2(n-1)}
\left(-\frac{n-1}{r^2}-\frac{1}{2}\right)
\frac{1}{g'}
-
\frac{2n-3}{2(n-1)}
\left(\frac{n-1}{r}-\frac{r}{2}\right)
\frac{g''}{(g')^2}
-\frac{1}{4}g'\\
&\geq0.
\end{aligned}
\end{equation*}
Indeed, after integrating from the initial point $R$ to $r$, and using $g(R)=0$ and $g'(R)=C>0$, we obtain \eqref{eq:L2-key-inequality}.

Multiplying the left-hand side of the above inequality by $(g')^2>0$, we get
\begin{equation*}
\begin{aligned}
&
\frac{2n-3}{2(n-1)}
\left(-\frac{n-1}{r^2}-\frac{1}{2}\right)g'
-
\frac{2n-3}{2(n-1)}
\left(\frac{n-1}{r}-\frac{r}{2}\right)g''
-\frac{1}{4}(g')^3\\
&=
\frac{2n-3}{2(n-1)}
\left(-\frac{n-1}{r^2}-\frac{1}{2}\right)g'\\
&\quad
-
\frac{2n-3}{2(n-1)}
\left(\frac{n-1}{r}-\frac{r}{2}\right)
\left(\frac{r}{2}-\frac{n-1}{r}\right)g'\\
&\quad
-
\frac{2n-3}{2(n-1)}
\left(\frac{n-1}{r}-\frac{r}{2}\right)
\left(\frac{r}{2}-\frac{n-1}{r}\right)(g')^3\\
&\quad
+
\frac{2n-3}{2(n-1)}
\left(\frac{n-1}{r}-\frac{r}{2}\right)
(1+(g')^2)\frac{g}{2}
-\frac{1}{4}(g')^3.
\end{aligned}
\end{equation*}
The term containing $g$ is non-negative. Therefore, it suffices to make the coefficients of $g'$ and $(g')^3$ non-negative. For the coefficient of $g'$, we need
\begin{equation}\label{eq:L2-condition-one}
\begin{aligned}
&
\frac{2n-3}{2(n-1)}
\left(-\frac{n-1}{r^2}-\frac{1}{2}\right)\\
&\quad
-
\frac{2n-3}{2(n-1)}
\left(\frac{n-1}{r}-\frac{r}{2}\right)
\left(\frac{r}{2}-\frac{n-1}{r}\right)
\geq0.
\end{aligned}
\end{equation}
For the coefficient of $(g')^3$, we need
\begin{equation}\label{eq:L2-condition-two}
\begin{aligned}
&
-\frac{2n-3}{2(n-1)}
\left(\frac{n-1}{r}-\frac{r}{2}\right)
\left(\frac{r}{2}-\frac{n-1}{r}\right)
-\frac{1}{4}
\geq0.
\end{aligned}
\end{equation}
Solving \eqref{eq:L2-condition-one} gives
\[
r\leq r_1
\overset{\mathrm{def}}{=}
\sqrt{
2\left[
n-\frac{1}{2}
-\sqrt{2n-\frac{7}{4}}
\right]
}.
\]
Solving \eqref{eq:L2-condition-two} gives
\[
r\leq r_2
\overset{\mathrm{def}}{=}
\sqrt{
2\left[
\frac{(n-1)(4n-5)}{2(2n-3)}
-
(n-1)
\sqrt{
\left(\frac{4n-5}{4n-6}\right)^2-1
}
\right]
}.
\]
Together with the additional restriction $0<r\leq r_3\overset{\mathrm{def}}{=}\sqrt{2(n-1)\frac{n-2}{n}}$, we finally obtain
\begin{equation*}
r\leq r_{L_2}
\overset{\mathrm{def}}{=}
\min\{r_1,r_2,r_3\}.
\end{equation*}
Thus $\eta\geq0$ holds in the region $0<r\leq r_{L_2}$. By Theorem~\ref{thm:horizontal-point-monotonicity}, the curve $L_2$ is monotone increasing in this region.
\end{proof}

\bibliographystyle{plain}
\bibliography{reference}

\end{document}